\documentclass[10pt]{amsart}
\usepackage{amssymb}
\usepackage{verbatim}
\usepackage[usenames]{color}
\usepackage{hyperref}
\usepackage{url}
\usepackage{mathrsfs}
\begin{document}

\newcommand{\s}{\sigma}
\newcommand{\al}{\alpha}
\newcommand{\om}{\omega}
\newcommand{\be}{\beta}
\newcommand{\la}{\lambda}
\newcommand{\vp}{\varphi}

\newcommand{\bo}{\mathbf{0}}
\newcommand{\bone}{\mathbf{1}}

\newcommand{\sse}{\subseteq}
\newcommand{\contains}{\supseteq}
\newcommand{\forces}{\Vdash}

\newcommand{\FIN}{\mathrm{FIN}}
\newcommand{\Fin}{\mathrm{Fin}}
\newcommand{\cf}{\mathrm{cf}}
\newcommand{\add}{\mathrm{add}}

\newcommand{\ve}{\vee}
\newcommand{\w}{\wedge}
\newcommand{\bv}{\bigvee}
\newcommand{\bw}{\bigwedge}
\newcommand{\bcup}{\bigcup}
\newcommand{\bcap}{\bigcap}

\newcommand{\rgl}{\rangle}
\newcommand{\lgl}{\langle}
\newcommand{\lr}{\langle\ \rangle}
\newcommand{\re}{\restriction}

\newcommand{\bB}{\mathbb{B}}
\newcommand{\bP}{\mathbb{P}}
\newcommand{\bR}{\mathbb{R}}
\newcommand{\bW}{\mathbb{W}}
\newcommand{\bX}{\mathbb{X}}
\newcommand{\bN}{\mathbb{N}}
\newcommand{\bQ}{\mathbb{Q}}
\newcommand{\bS}{\mathbb{S}}
\newcommand{\St}{\tilde{S}}
   
\newcommand{\sd}{\triangle}
\newcommand{\cl}{\prec}
\newcommand{\cle}{\preccurlyeq}
\newcommand{\cg}{\succ}
\newcommand{\cge}{\succcurlyeq}
\newcommand{\dom}{\mathrm{dom}\,}

\newcommand{\lra}{\leftrightarrow}
\newcommand{\ra}{\rightarrow}
\newcommand{\llra}{\longleftrightarrow}
\newcommand{\Lla}{\Longleftarrow}
\newcommand{\Lra}{\Longrightarrow}
\newcommand{\Llra}{\Longleftrightarrow}
\newcommand{\rla}{\leftrightarrow}
\newcommand{\lora}{\longrightarrow}
\newcommand{\E}{\mathrm{E}}
\newcommand{\rank}{\mathrm{rank}}
\newcommand{\lefin}{\le_{\mathrm{fin}}}
\newcommand{\Ext}{\mathrm{Ext}}
\newcommand{\Erdos}{Erd{\H{o}}s}
\newcommand{\Pudlak}{Pudl{\'{a}}k}
\newcommand{\Rodl}{R{\"{o}}dl}
\newcommand{\Juhasz}{Juh{\'{a}}sz}
\newcommand{\Todorcevic}{Todor{\v{c}}evi{\'{c}}}
\newcommand{\Pospisil}{Posp{\'{i}}{\v{s}}il}
\newcommand{\Fraisse}{Fra{\"{i}}ss{\'{e}}}
\newcommand{\Nesetril}{Ne{\v{s}}et{\v{r}}il}

\newtheorem{thm}{Theorem}  
\newtheorem{prop}[thm]{Proposition} 
\newtheorem{lem}[thm]{Lemma} 
\newtheorem{cor}[thm]{Corollary} 
\newtheorem{fact}[thm]{Fact}    
\newtheorem{facts}[thm]{Facts}      
\newtheorem*{thmMT}{Main Theorem}
\newtheorem*{thmMTUT}{Main Theorem for $\vec{\mathcal{U}}$-trees}
\newtheorem*{thmnonumber}{Theorem}
\newtheorem*{mainclaim}{Main Claim}

\theoremstyle{definition}   
\newtheorem{defn}[thm]{Definition} 
\newtheorem{example}[thm]{Example} 
\newtheorem{conj}[thm]{Conjecture} 
\newtheorem{prob}[thm]{Problem} 
\newtheorem{examples}[thm]{Examples}
\newtheorem{question}[thm]{Problem}
\newtheorem{problem}[thm]{Problem}
\newtheorem{openproblems}[thm]{Open Problems}
\newtheorem{openproblem}[thm]{Open Problem}
\newtheorem{conjecture}[thm]{Conjecture}
\newtheorem*{problem1}{Problem 1}
\newtheorem*{problem2}{Problem 2}
\newtheorem*{problem3}{Problem 3}
\newtheorem*{notation}{Notation}

\theoremstyle{remark} 
\newtheorem*{rem}{Remark} 
\newtheorem*{rems}{Remarks} 
\newtheorem*{ack}{Acknowledgments} 
\newtheorem*{note}{Note}
\newtheorem*{claim}{Claim}
\newtheorem*{claim1}{Claim $1$}
\newtheorem*{claim2}{Claim $2$}
\newtheorem*{claim3}{Claim $3$}
\newtheorem*{claim4}{Claim $4$}
\newtheorem{claimn}{Claim}
\newtheorem{subclaim}{Subclaim}
\newtheorem*{subclaimnn}{Subclaim}
\newtheorem*{subclaim1}{Subclaim (i)}
\newtheorem*{subclaim2}{Subclaim (ii)}
\newtheorem*{subclaim3}{Subclaim (iii)}
\newtheorem*{subclaim4}{Subclaim (iv)}
\newtheorem{case}{Case}
 \newtheorem*{case1}{Case 1}
\newtheorem*{case2}{Case 2}
\newtheorem*{case3}{Case 3}

\newcommand{\noprint}[1]{\relax}
\newenvironment{nd}{\noindent\color{red}Note to Stevo: }{}
%To get rid of AB comments, replace {ab} below (not above) this line
%by {comment}\noprint1 and 
%observe that the replacement is reversible.  

\title{Survey on the Tukey theory of ultrafilters}
\author{Natasha Dobrinen}
\address{Department of Mathematics\\
 University of Denver \\
2360 Gaylord St.\\ Denver, CO \ 80208 U.S.A.}
\email{natasha.dobrinen@du.edu}
\urladdr{http://web.cs.du.edu/~ndobrine}
%\thanks{This work  was partially supported by a grant from the Simons Foundation}

\begin{abstract}
This article surveys  results regarding the  Tukey theory of ultrafilters on  countable base sets.
The driving forces for this investigation are 
Isbell's Problem and the question of how closely related the  Rudin-Keisler  and Tukey
reducibilities are.
We review work on the possible structures of cofinal types
 and
conditions which guarantee that an ultrafilter is  below the Tukey maximum.
The known canonical forms for cofinal maps on ultrafilters are reviewed,
as well as their applications to finding 
which structures
 embed into the Tukey types of ultrafilters.
With the addition of some Ramsey theory,
fine analyses of the structures at the bottom of the Tukey hierarchy are made.
\end{abstract}

%\begin{keyword}
%ultrafilter \sep Tukey \sep cofinal map\sep finitary %approximation maps \sep partition property
%\MSC Primary: 54D80 \sep 03E04; Secondary:  03E05 
%\end{keyword}
%\end{frontmatter}

\maketitle

\section{Introduction and  Brief History of Tukey Reducibility}\label{intro}

In \cite{Tukey40}, Tukey introduced the Tukey ordering  to 
develop the notion of Moore-Smith convergence in topology.
After its initial success  in helping develop general topology,
Tukey reducibility was studied in its own right as a means for comparing partially ordered sets.
It was employed by Day in \cite{Day44}, Isbell in \cite{Isbell65} and \cite{Isbell72},  and \Todorcevic\ in \cite{TodorcevicDirSets85} and \cite{Todorcevic96}  to find rough classifications of partially ordered sets, a setting where isomorphism is too fine  a notion to give any reasonable classification.
More recently, Tukey reducibility has proved to be foundational in the study of analytic partial orderings (see \cite{Solecki/Todorcevic04}, \cite{Solecki/Todorcevic11} and \cite{Matrai10}).
This background will be discussed in some more detail below.

When restricting attention to the class of ultrafilters on a countable base set, 
Tukey reducibility turns out to be a coarsening of the well-studied Rudin-Keisler reducibility.
In the last several years, this topic has been the focus of much work.
In this article, we survey the known results and continuing work regarding the structure of the Tukey types of ultrafilters.  In each section we shall point out the current outstanding questions related to the topic of that section.  As there are too many interesting open problems to list, we will  point out those questions which guide the rest of the study of Tukey types of ultrafilters.

We now begin the introduction of Tukey reducibility for  general partial orderings, and in particular, for ultrafilters.
A pair $(D,\le)$ is a {\em partially ordered set} or  {\em partial ordering} if $D$ is a set and $\le$ is a relation which is 
reflexive, antisymmetric and transitive.
A partial order $\le_D$ on a set $D$ is {\em directed} if for any two members $d_1,d_2\in D$, there is another member $d_3\in D$ such that $d_3$ is above both $d_1$ and $d_2$; that is, $d_3\ge_D d_1$ and $d_3\ge d_2$.
A subset $X\sse D$ of a partially ordered set $(D,\le_D)$ is {\em cofinal} in $D$ if for every $d\in D$ there is an $x\in X$ such that $d\le_D x$.
Let $(D,\le_D)$ and $(E,\le_E)$ be partial orderings.
We say that a function $f:E\ra D$ is {\em cofinal}  if the image of each cofinal subset of $E$ is cofinal in $D$;
that is, for each cofinal subset $Y\sse E$, its $f$-image, which we denote $f'' Y:=\{f(e):e\in Y\}$, is a cofinal subset of $D$.
\begin{defn}\label{def.Tukeyred}
We say that a partial ordering $(D,\le_D)$ is {\em Tukey reducible} to a partial ordering  $(E,\le_E)$, and write $D\le_T E$, if there is a cofinal map from $(E,\le_E)$ to $(D,\le_D)$.
\end{defn}

An equivalent formulation of Tukey reducibility was noticed by Schmidt in \cite{Schmidt55}.
A subset $X\sse D$ of a partially ordered set $(D,\le_D)$ is called {\em unbounded} if there is no single $d\in D$ which simultaneously bounds every member of $X$; that is 
for each
 $d\in D$, there is some $x\in X$  such that $d\not\ge_D x$.
A map
$g:D\ra E$ is called a {\em Tukey map} or an {\em unbounded map} if the $g$-image of each unbounded subset of $D$ is an unbounded subset of $E$; that is, whenever $X\sse D$ is unbounded in $(D,\le_D)$, then $d'' X$ is unbounded in $(E,\le_E)$.
\begin{fact}[Schmidt, \cite{Schmidt55}]
For partially ordered sets $(D,\le_D)$ and $(E\le_E)$, 
$E\ge_T D$ if and only if there is a Tukey map from $D$ into $E$.
\end{fact}

If both $D\le_T E$ and $E\le_T D$, then we write $D\equiv_T E$ and say that $D$ and $E$ are Tukey equivalent.
The relation
$\equiv_T$ is an equivalence relation, and $\le_T$ on the equivalence classes forms a partial ordering. 
The equivalence classes are called {\em Tukey types}, and  are themselves partially ordered by the  Tukey reduction $\le_T$.

For directed partial orders, the notion of Tukey equivalence coincides with the notion of cofinal similarity.
Two partial orders $D$ and $E$ are {\em cofinally similar} if there is another partial order into which they both embed as cofinal subsets. 
It turns out that two directed partial orders are Tukey equivalent if and only if they are cofinally similar (\cite{Tukey40}; see also \cite{TodorcevicDirSets85}).
This provides an intuitive way of thinking about Tukey equivalence for directed partial orders.
(See \cite{TodorcevicDirSets85} for further details.)

Tukey reducibility captures  important aspects of convergent functions and of the structure  of  neighborhood bases in  topological spaces.
In the theory of convergence in general topology, given a net, the domain of any subnet is a cofinal (hence, Tukey equivalent) subset of the domain of the net.
Moreover, Tukey reducibility distills  essential properties of continuity, as is shown in the following 
 Theorem 2.1 in \cite{Milovich08}.
Suppose $X$ and $Y$ are topological spaces, $p\in X$, $q\in Y$, $\mathcal{A}$ is a local base at $p$ in $X$, $\mathcal{B}$ is a local base at $q$ in $Y$, $f:X\ra Y$ is continuous and open (or just continuous at $p$ and open at $q$),
and $f(p)=q$.
Then $(\mathcal{B},\contains)\le_T(\mathcal{A},\contains)$.
It follows that if $f$ is a homeomorphism, then every local base at $p$ is Tukey equivalent to a local base at $q$.

The study of  Tukey types of partial orderings often reveals useful information for the comparison of different partial orderings.
For example, Tukey reducibility downward preserves calibre-like properties, such as the countable chain condition, property K, precalibre $\aleph_1$, $\sigma$-linked, and $\sigma$-centered.
It also preserves the cardinal characteristics of cofinality and additivity:
If $D\le_T E$, then $\cf(D)\le \cf(E)$ and $\add(D)\ge \add(E)$.
 (See \cite{Todorcevic96} for more  on that subject.)

Satisfactory classification theories of Tukey types have been developed for several classes of ordered sets.
The  cofinal types of countable directed partial orders are $1$ and $\om$ (see \cite{Tukey40}).
Day found a classification of
countable partially ordered sets in \cite{Day44} in terms of direct sums of three families of partially ordered sets.
For partial orders with size of the first uncountable cardinal, $\aleph_1$, the usual  axioms of set theory,  Zermelo-Fraenkel Axioms plus the Axiom of Choice (ZFC), are not sufficient to determine the Tukey structure.
The picture 
actually depends on what axioms one assumes in addition to ZFC.
Tukey showed in \cite{Tukey40} that the directed partial orders $1$, $\om$, $\om_1$, $\om\times\om_1$, and $[\om_1]^{<\om}$ are all cofinally distinct; that is, no two of them are Tukey equivalent.
He then asked whether there are more cofinal types of directed partial orders of cardinality $\aleph_1$.
In \cite{TodorcevicDirSets85}, \Todorcevic\  showed that there exist $2^{\aleph_1}$ Tukey inequivalent directed sets of cardinality continuum. So under CH, the answer to Tukey's question is that there exist $2^{\aleph_1}$ diferent cofinal types of directed sets of cardinality at most $\aleph_1$.
Assuming the Proper Forcing Axiom (PFA), (which is a sort of generalization of the Baire Category Theorem to more general partial orders), \Todorcevic\ in \cite{TodorcevicDirSets85} 
 classified the Tukey types of all 
directed partial orders of cardinality $\aleph_1$  by showing that there are exactly five cofinal types, namely the ones listed above.

In \cite{Todorcevic96}, under the assumption of PFA, \Todorcevic\ classified the Tukey types of all partially ordered sets of size $\aleph_1$
in terms of a countable basis constructed from five forms of partial orderings.
On the other hand, \Todorcevic\  also showed in \cite{Todorcevic96} that there are at least $2^{\aleph_1}$ many Tukey incomparable  separative $\sigma$-centered partial orderings of size continuum, the cardinality of the real numbers.
Thus,  no  satisfactory classification theory via Tukey reducibility of all partial orderings of size continuum is possible.

However, 
when we restrict to a particular class of partial orderings of size continuum, Tukey reducibility
can still yield useful information, especially in settings  where isomorphism is too strong for any structure to be found.  
One of the first of these lines of investigation is 
 in the paper \cite{Fremlin91} of  Fremlin who considered Tukey structure on partially ordered sets occurring in analysis.
After this, several papers appeared dealing with different classes of posets such as, for example, the paper \cite{Solecki/Todorcevic04} of Solecki and \Todorcevic\ which makes
a  systematic study of the  structure of the Tukey types of topological directed sets.

In recent years, much focus has been on the Tukey structure of ultrafilters.
The {\em power set} of the natural numbers, denoted $\mathcal{P}(\om)$, is the collection of all subsets of $\om$.
\begin{defn}\label{def.uf}
A {\em filter} on base set $\om$ is a set $\mathcal{F}\sse\mathcal{P}(\om)$ such that 
$\mathcal{F}$ is 
\begin{enumerate}
\item 
 not empty: $\om\in\mathcal{F}$;
\item
closed under intersections: whenever $X,Y\in\mathcal{F}$, then $X\cap Y\in\mathcal{F}$; and
\item 
closed upwards: whenever $X\in\mathcal{F}$, $Y\sse \om$ and $Y\contains X$,
then $Y\in\mathcal{F}$.
\end{enumerate} 
A filter $\mathcal{F}$ is {\em proper} if $\emptyset\not\in\mathcal{F}$, and is {\em nonprincipal} if it contains no singletons.
A set $\mathcal{U}\sse\mathcal{P}(\om)$ is an {\em ultrafilter } if $\mathcal{U}$ is a proper filter with the additional property that
for each $X\sse\om$, either $X\in \mathcal{U}$ or $\om\setminus X\in\mathcal{U}$. 
\end{defn}

Each ultrafilter on a countable base set has size continuum, the cardinality of the real numbers, denoted by $\mathfrak{c}$.
Except when necessary, we shall assume the countable base for our ultrafilters to be the set natural numbers $\{0,1,2,\dots\}$, denoted $\om$.
An ultrafilter $\mathcal{U}$ on $\om$ can be considered partially ordered by reverse inclusion $\contains$, or by reverse almost inclusion $\contains^*$.  In fact, $(\mathcal{U},\contains)$ is a directed partial ordering, since every ultrafilter is closed under finite intersections.
Likewise,  $(\mathcal{U},\contains)$ is a directed partial ordering.
The directed partial order $([\mathfrak{c}]^{<\om},\sse)$ is the Tukey maximum among all directed partial orders of size continuum.
That is, for every directed partial order $(D,\le_D)$ with $D$ having cardinality continuum, it is in fact true that $(D,\le_D)$ is Tukey reducible to $([\mathfrak{c}]^{<\om},\sse)$.

The study of Tukey reducibility on ultrafilters on countable base sets began with  \cite{Isbell65}, where it was shown that there is an ultrafilter which achieves the maximal Tukey type.

\begin{thm}[Isbell, Theorem 5.4 in \cite{Isbell65}]\label{thm.3}
There is an ultrafilter $\mathcal{U}_{\mathrm{top}}$ on $\om$ realizing the maximal cofinal type among all directed sets of cardinality continuum, i.e.\ $(\mathcal{U}_{\mathrm{top}},\contains)\equiv_T([\mathfrak{c}]^{<\om},\sse)$.
\end{thm}
One can actually trace back the seed of the idea of an  ultrafilter with maximum Tukey type  to work of \Pospisil\ in \cite{Pospisil39}
(see also \cite{Juhasz66} and  \cite{Juhasz67}).

One of the main guiding forces in the study of Tukey types of ultrafilters is the following problem of Isbell.

\begin{problem}[Isbell, Problem 2 in \cite{Isbell65}]\label{prob.Isbell}
It there an ultrafilter $\mathcal{U}$ on $\om$ such that $(\mathcal{U},\contains)<_T([\mathfrak{c}]^{<\om},\sse)$?
\end{problem}

Interest in Isbell's Problem  problem and more generally in the study of Tukey types of ultrafilters was  revived
by work of Milovich in  \cite{Milovich08}.  In that paper, he showed that, using the combinatorial principle $\lozenge$ in addition to ZFC, one can construct a nonprincipal ultrafilter strictly below the Tukey maximum.
A few years later, Dobrinen and \Todorcevic\  showed  in \cite{Dobrinen/Todorcevic11} that  every ultrafilter which is a p-point has Tukey type strictly below the maximum type, thus reducing the extra axioms needed to produce a non-top ultrafilter.
In the past several years,  
there has been a a large body of research on  ultrafilters which have Tukey type  below the maximum,
and the major  results will be discussed in this survey.
However, it is still unknown whether it is consistent with ZFC that every non-principal ultrafilter has the maximum Tukey type.
The following is what is currently left open of Isbell's Problem.

\begin{problem}\label{prob.Isbellleft}
Is it consistent with ZFC that every nonprincipal ultrafilter has maximum Tukey type?
\end{problem}

Any model of ZFC in which there is only one Tukey type of nonprincipal ultrafilter on $\om$ must be a model in which there are no p-points,
as will be shown in Section \ref{sec.<top}.

Strongly connected with this problem is the following.

\begin{problem}\label{prob.nottop}
What properties suffice to guarantee that an ultrafilter has Tukey type strictly below the maximum?
\end{problem}

Answers to Problem \ref{prob.nottop} are found throughout this survey, though more work is still needed.

The second main guiding force for the study of Tukey types of ultrafilters is its connection with the notion of Rudin-Keisler reducibility on ultrafilters, which is now reviewed.
\begin{defn}\label{defn.RK}
Let $\mathcal{U}$ and $\mathcal{V}$ be ultrafilters on $\om$.
For a function $f:\om\ra\om$, define $f(\mathcal{U})=\{X\sse\om:f^{-1}(X)\in\mathcal{U}\}$.
We say that $\mathcal{V}$ is {\em Rudin-Keisler reducible} to $\mathcal{U}$,
denoted $\mathcal{V}\le_{RK}\mathcal{U}$, if there is a function $f:\om\ra\om$ such that $f(\mathcal{V})=\mathcal{U}$.
\end{defn}
We say that two ultrafilters $\mathcal{U}$ and $\mathcal{V}$ are Rudin-Keisler equivalent, denoted $\mathcal{U}\equiv_{RK}\mathcal{V}$, if and only if they are both Rudin-Keisler reducible to each other.
It is well-known that two ultrafilters are Rudin-Keisler equivalent if and only if they are isomorphic, meaning that there is a bijection $f:\om\ra\om$ such that $\mathcal{U}=f(\mathcal{V})$.
This justifies using the terminology of {\em Rudin-Keisler type} and {\em isomorphism class} interchangeably.

\begin{fact}\label{fact.TRK}
Let $\mathcal{U}$ and $\mathcal{V}$ be ultrafilters on $\om$.
If $\mathcal{U}\ge_{RK} \mathcal{V}$, then $\mathcal{U}\ge_T\mathcal{V}$.
\end{fact}

\begin{proof}
Take a function $h:\om\ra\om$ satisfying $\mathcal{V}=h(\mathcal{U}):=\{X\sse\om:h^{-1}(X)\in\mathcal{U}\}$.  
Define $f:\mathcal{U}\ra\mathcal{V}$ by $f(X)=\{h(n):n\in X\}$, for each $X\in\mathcal{U}$.
Then $f$ is a cofinal map. 
\end{proof}

Thus,
 Tukey reducibility is a natural weakening of Rudin-Keisler reducibility, and consequently, Tukey equivalence classes
 (also called Tukey types) are a coarsening of isomorphism classes.

The Rudin-Keisler order on ultrafilters has been well-studied for several decades.  This line of research has proved useful in several areas of mathematics.
Recall that  the  points in the Stone-\v{C}ech compactification of the natural numbers, denoted $\beta \om$, can be identified with collection of ultrafilters on $\om$.;
the Stone-\v{C}ech remainder $\beta\om\setminus\om$ can be identified with the collection of nonprincipal ultrafilters on $\om$.
Thus,
Rudin-Keisler reducibility for ultrafilters on $\om$ provides a means for classifying  the points of $\beta\om$ and  $\beta\om\setminus\om$.
Rudin-Keisler reducibility also has important  model-theoretic  implications.
If  $\mathcal{U}\ge_{RK}\mathcal{V}$, then the ultraproduct (of any structure) by $\mathcal{V}$ elementarily embeds into the ultraproduct by $\mathcal{U}$.
As the Rudin-Keisler order on ultrafilters has had such an impact in these and other areas of mathematics, and as Tukey reducibility is a natural weakening of it, 
the following natural question takes on importance.

\begin{problem}\label{prob.TvsRK}
How different are Tukey and Rudin-Keisler reducibility?
\end{problem}

Problems \ref{prob.Isbellleft},
\ref{prob.nottop}
and \ref{prob.TvsRK}
are the driving forces behind all current investigations of the structure of the Tukey types of ultrafilters.
This will be seen throughout this article.

%*****************************************************************

\section{Definitions, facts, and some special types of ultrafilters}\label{sec.thebasics}

This section begins the survey  by recounting basic definitions and some easy but very useful facts about Tukey reductions.
Their proofs can be found in Section 2 of \cite{Dobrinen/Todorcevic11}.

\begin{fact}\label{fact.cofinalTukeysame}
If $C$ is a cofinal subset of a partial ordering $(P,\le)$,
then $(C,\le)\equiv_T(P,\le)$.
\end{fact}

\begin{defn}\label{defn.monotonemap}
Let $(P,\le_P)$ and $(Q,\le_Q)$ be partial orderings.
A  map $f:P\ra Q$ is {\em monotone}
if whenever $p,r$ are in $P$ and $p\le_P r$, then $f(p)\le_Q f(r)$.
For the special case of ultrafilters $\mathcal{U},\mathcal{V}$, this translates to the following:
A map $f:\mathcal{U}\ra\mathcal{V}$ is {\em monotone} if 
whenever $W,X\in\mathcal{U}$ and $W\contains X$, then $f(W)\contains f(X)$.  
\end{defn}

\begin{fact}\label{fact.monotonemap}
Let $(P,\le_P)$ and $(Q,\le_Q)$ be partial orderings.
A monotone map $f:P\ra Q$ is a cofinal map if and only if its image $f''P$ is a cofinal subset of $Q$.
\end{fact}

Recall that an ultrafilter $\mathcal{U}$ is partially ordered by $\contains$, and moreover,  $(\mathcal{U},\contains)$ is a directed partial ordering.
The next fact shows that, for ultrafilters, we only need consider monotone cofinal maps.

\begin{fact}\label{fact.monotone}
Let  $\mathcal{U}$ and $\mathcal{V}$ be ultrafilters.
If $\mathcal{U}\ge_T\mathcal{V}$, then this is witnessed by a monotone cofinal map. 
\end{fact}

A standard method of building new ultrafilters from old ones is the Fubini product construction.
The construction can be iterated countably many times, each time producing another ultrafilter.

\begin{notation}\label{notation.cross}
Let $\mathcal{U}$, $\mathcal{V}$, and $\mathcal{U}_n$ ($n<\om$) be ultrafilters.
We define the notation for the following ultrafilters. 
\begin{enumerate}
\item
The {\em Fubini product} of the $\mathcal{U}_n$ over $\mathcal{U}$ is  
$$\lim_{n\ra\mathcal{U}}\mathcal{U}_n=\{A\sse\om\times\om:\{n\in\om:\{j\in\om:(n,j)\in A\}\in\mathcal{U}_n\}\in\mathcal{U}\}.$$
\item
When all $\mathcal{U}_n$ are the same, say $\mathcal{V}$, then we denote $\lim_{n\ra\mathcal{U}}\mathcal{V}$ by $\mathcal{U}\cdot\mathcal{V}$.
In this case,
$$\mathcal{U}\cdot\mathcal{V}= \{A\sse\om\times\om:\{i\in\om:\{j\in\om:(i,j)\in A\}\in\mathcal{V}\}\in\mathcal{U}\}.$$
\end{enumerate}
\end{notation}

The cartesian product of ultrafilters forms again a directed partial ordering, though not an ultrafilter.
Let
$\mathcal{U}\times\mathcal{V}$ denote the ordinary cartesian product of $\mathcal{U}$ and $\mathcal{V}$ with the coordinate-wise ordering $\lgl\contains,\contains\rgl$;
and let
$\Pi_{n<\om}\mathcal{U}_n$ is the cartesian product of the $\mathcal{U}_n$ with its natural coordinate-wise product ordering.

Cartesian products of two ultrafilters produce the least upper bound of them in the Tukey ordering.
For Fubini products, the relationship is not so straightforward. 
In some instances it is known that 
the Fubini product of two ultrafilters $\mathcal{U}$ and $\mathcal{V}$ is the Tukey least upper bound of them, and this is pointed out below.  
It is not known whether this is always the case.
The following gives the basic facts on the Tukey  relationships between cartesian and Fubini products of ultrafilters.
Proofs can be found in Fact 7 in \cite{Dobrinen/Todorcevic11}.

\begin{fact}
Let $\mathcal{U},\mathcal{U}_0,\mathcal{U}_1,\mathcal{V},\mathcal{V}_0$, and $\mathcal{V}_1$ be ultrafilters.
\begin{enumerate}
\item
$\mathcal{U}\times\mathcal{U}\equiv_T\mathcal{U}$.
\item
$\mathcal{U}\times \mathcal{V}\ge_T\mathcal{U}$ and 
$\mathcal{U}\times \mathcal{V}\ge_T\mathcal{V}$.
\item
If $\mathcal{U}_1\ge_T \mathcal{U}_0$ and $\mathcal{V}_1\ge_T\mathcal{V}_0$, then $\mathcal{U}_1\times\mathcal{V}_1\ge_T\mathcal{U}_0\times\mathcal{V}_0$.
\item
If $\mathcal{W}\ge_T\mathcal{U}$ and $\mathcal{W}\ge_T\mathcal{V}$,
then $\mathcal{W}\ge_T\mathcal{U}\times\mathcal{V}$.
Thus, $\mathcal{U}\times\mathcal{V}$ is the minimal Tukey type which is Tukey greater than or equal to both $\mathcal{U}$ and $\mathcal{V}$.
\item
 $\mathcal{U}\cdot\mathcal{V}\ge_T \mathcal{U}$ and
 $\mathcal{U}\cdot\mathcal{V}\ge_T \mathcal{V}$,
 and therefore $\mathcal{U}\cdot\mathcal{V}\ge_T\mathcal{U}\times\mathcal{V}$. 
\end{enumerate}
\end{fact}

\begin{rem}\label{rem.dots}
One cannot conclude from the above that $\mathcal{U}\cdot\mathcal{V}\equiv_T\mathcal{U}\times\mathcal{V}$.
More will be said about this in  Section \ref{sec.Fubprod}, especially in
Theorem \ref{thm.Milovich5.2}.
\end{rem}

Next, we review the definitions of important types of ultrafilters. (1) - (5)  can found in \cite{Bartoszynski/JudahBK}; (6) is found in \cite{Kunen78}.
Recall the standard notation $\sse^*$,
where for $X,Y\sse\om$, we write
$X\sse^* Y$ to denote that $|X\setminus Y|<\om$.

\begin{defn}\label{defn.uftypes}
Let $\mathcal{U}$ be a nonprincipal ultrafilter.
\begin{enumerate}
\item
$\mathcal{U}$ is {\em selective} if 
for every function $f:\om\ra\om$,
there is an $X\in\mathcal{U}$ such that either $f\re X$ is constant or $f\re X$ is one-to-one.
\item
$\mathcal{U}$ is {\em Ramsey} if for each $2$-coloring $f:[\om]^2\ra 2$, there is  an $X\in\mathcal{U}$ such that $f\re [X]^2$ takes on exactly one color.
\item
$\mathcal{U}$ is a {\em p-point} if for every family $\{X_n:n<\om\}\sse\mathcal{U}$
there is an $X\in\mathcal{U}$ such that $X\sse^* X_n$ for each $n<\om$.
\item
$\mathcal{U}$ is a {\em q-point} if for each partition of $\om$ into finite pieces $\{I_n:n<\om\}$,
there is an $X\in\mathcal{U}$ such that $|X\cap I_n|\le 1$ for each $n<\om$.
\item
$\mathcal{U}$ is {\em rapid} if 
for each function $f:\om\ra\om$,
there exists an $X\in\mathcal{U}$ such that $|X\cap f(n)|\le n$ for each $n<\om$.
\item
$\mathcal{U}$ 
is {\em  $\kappa$-OK} if whenever $U_n\in\mathcal{U}$ ($n<\om$),
there is a $\kappa$-sequence $\lgl V_{\al}:\al<\kappa\rgl$ of elements of $\mathcal{U}$ such that for all $n\ge 1$,
for all $\al_1<\dots<\al_n<\kappa$,
$V_{\al_1}\cap\dots\cap V_{\al_n}\sse^* U_n$.
\end{enumerate}
\end{defn}

 $\mathfrak{c}$-OK ultrafilters always exist in ZFC, hence so do $\kappa$-OK points, for all $\om_1\le\kappa\le\mathfrak{c}$.
All of the ultrafilters  in (1) - (5) exist when either the Continuum Hypothesis (CH) or  Martin's Axiom (MA)  holds.
In fact, much weaker cardinal invariant assumptions suffice to ensure their existence. 
However, the existence of selective ultrafilters, p-points, q-points, or even rapid ultrafilters does not follow from ZFC.
We refer the interested reader to \cite{Bartoszynski/JudahBK}, \cite{BlassHB}  and \cite{Kunen78} for further exposition on these topics.

The following well-known implications (1) and (2) can  be found in \cite{Bartoszynski/JudahBK},
while (3) can be found in \cite{Kunen78}.

\begin{thm}
\begin{enumerate}
\item
An ultrafilter is  selective if and only if
it is Ramsey
if and only if
 it is both a p-point and  a q-point.
\item
Every q-point is rapid.
\item
If $\mathcal{U}$ is $\kappa$-OK and $\kappa>$ cof$(\mathcal{U})$, then $\mathcal{U}$ is a p-point.
\end{enumerate}
\end{thm}

Another class of ultrafilters considered are ordered union  ultrafilters on base set $\FIN$, which
 denotes the collection of nonempty finite subsets of $\om$.
Ultrafilters on this base set have been studied by Glazer (see \cite{ComfortBK74}) and later by Blass in \cite{Blass87}.
The availability of  Hindman's Theorem on $\FIN$ (see \cite{Hindman74}) enables the construction of ultrafilters which have strong Ramsey properties.
Here we only provide the basics, leaving the interested reader with the resources \cite{ComfortBK74}, \cite{Blass87} and \cite{Dobrinen/Todorcevic11} for further information.

Let
$\min:\FIN\ra\om$ and  $\max:\FIN\ra\om$
denote the maps which take each non-empty finite set $x\in\FIN$ to its minimal and maximal elements, respectively.
The map $(\min,\max):\FIN\ra\om\times\om$  is defined by $(\min,\max)(x)=(\min(x),\max(x))$.
Whenever $\mathcal{U}$ is an ultrafilter on $\FIN$,
then 
$\mathcal{U}_{\min}$
 ($\mathcal{U}_{\max}$)  is an ultrafilter on $\om$
 generated by the collection of sets $\{\min(x):x\in U\}$,  $U\in\mathcal{U}$
($\{\max(x):x\in U\}$, $U\in\mathcal{U}$).
$\mathcal{U}_{(\min,\max)}$ is the ultrafilter on $\om\times\om$ generated by
$\{(\min(x),\max(x)):x\in U\}$, $U\in\mathcal{U}$.
Moreover,
$\mathcal{U}\ge_{RK}\mathcal{U}_{\min,\max}$,
$\mathcal{U}_{\min,\max}\ge_{RK}\mathcal{U}_{\min}$, and
$\mathcal{U}_{\min,\max}\ge_{RK}\mathcal{U}_{\max}$.
Thus, the same Tukey reductions between these ultrafilters hold.

\begin{defn}\label{def.block-gen}
A {\em block-sequence} of $\FIN$ is an infinite sequence $X=(x_n)_{n<\om}$ of elements of  $\FIN$ such that for each $n<\om$,
$\max(x_n)<\min(x_{n+1})$.
For a block-sequence $X$,
we let $[X]$ denote  $\{x_{n_1}\cup\dots\cup x_{n_k}:k<\om$ and $n_1<\dots <n_k\}$, the set of finite unions of elements of $X$.
For any $m<\om$,
let $X/m$ denote $(x_n)_{n\ge k}$ where $k$ is least such that $\min(x_k)\ge m$.

An (idempotent) ultrafilter $\mathcal{U}$ on base set $\FIN$ is called {\em block generated}
if it is generated by the set of $[X]$, where $X$ is some collection of block sequences.
Such ultrafilters are called 
{\em ordered-union } in \cite{Blass87}.
\end{defn}

We now state some facts about block-generated ultrafilters.

\begin{fact}\label{facts.beginning}
Let $\mathcal{U}$ be any nonprincipal block-generated ultrafilter on $\FIN$.
\begin{enumerate}
\item
(Corollary 3.6, \cite{Blass87})
$\mathcal{U}$ is not a p-point.
\item
$\mathcal{U}$ is not a q-point.
\item
(Corollary 3.7, \cite{Blass87})
$\mathcal{U}_{\min,\max}$ is isomorphic (i.e.\ Rudin-Keisler equivalent) to $\mathcal{U}_{\min}\cdot\mathcal{U}_{\max}$.
\item
(Proposition 3.9 \cite{Blass87})
$\mathcal{U}_{\min}$ and $\mathcal{U}_{\max}$
are q-points.
\item 
(Fact 65 in \cite{Dobrinen/Todorcevic11})
$\mathcal{U}_{\min,\max}$ is neither a p-point nor a q-point.
\item
(Fact 65 in \cite{Dobrinen/Todorcevic11})
If $\mathcal{U}_{\min}$ is selective, then $\mathcal{U}_{\min,\max}$ is rapid.
\end{enumerate}
\end{fact}

By (5),  the existence of block-generated ultrafilters on $\FIN$
 cannot be proved on from
ZFC,
though using Hindman's Theorem one can easily establish the existence of such ultrafilters using CH or MA.

%***************************************************************************

\section{Basic Differences between Tukey and Rudin-Keisler}\label{sec.Fubprod}

As we pointed out in the Introduction,
Tukey reducibility is a generalization of Rudin-Keisler reducibility: Whenever $\mathcal{U}\ge_{RK}\mathcal{V}$, then also $\mathcal{U}\ge_T\mathcal{V}$.
The Tukey equivalence class of an ultrafilter $\mathcal{U}$ is the collection of all ultrafilters cofinally equivalent to $\mathcal{U}$.
The Rudin-Keisler equivalence class of $\mathcal{U}$ is the collection of all ultrafilters $\mathcal{V}$ isomorphic to $\mathcal{U}$, meaning that there is a permutation $f:\om\ra\om$ such that $\mathcal{V}=f(\mathcal{U})$.
Thus, the  Tukey equivalence class of an ultrafilter is composed of isomorphism classes.
Since Tukey reducibility is defined using cofinal maps on an ultrafilter, 
 a priori, there are $2^{\mathfrak{c}}$ such maps.
Thus,  it seems at first glance that the Tukey type of an ultrafilter has cardinality $2^{\mathfrak{c}}$.
We do know this to be the case for the maximum Tukey type, as the construction in 
Theorem \ref{thm.3}
produces $2^{\mathfrak{c}}$ ultrafilters, each with maximum Tukey type.
Juxtaposing this, every Rudin-Keisler equivalence class of an ultrafilter has cardinality $\mathfrak{c}$, as there are exactly continuum many permutations of $\om$.
Thus, the maximum Tukey type is composed of $2^{\mathfrak{c}}$ isomorphism classes of ultrafilters.
In particular, this provides an example (in ZFC) of a Tukey type with cardinality strictly greater than the cardinality of any Rudin-Keisler type.

This sets the stage for the basic differences between Tukey and Rudin-Keisler equivalence relations, and moreover, for studying the structure of the isomorphism classes within a Tukey class.
We shall see  in Section \ref{sec.canonical.maps} that sometimes Tukey types have cardinality $\mathfrak{c}$.
However, the existence of any of the known examples of ultrafilters with Tukey type of cardinality $\mathfrak{c}$ requires axioms additional to ZFC, and moreover, their Tukey types still contain at least $\om_1$ many isomorphism classes, as shall be seen in Section \ref{sec.Ramsey}.

Another easily seen difference between Rudin-Keisler and Tukey reducibility is that,
while there is a maximum Tukey type, there is 
no maximal Rudin-Keisler class of ultrafilters.
The Fubini product construction makes this easy to see, as
given any non-isomorphic ultrafilters $\mathcal{U},\mathcal{V}$, it is always the case that $\mathcal{U}\cdot\mathcal{V}>_{RK}\mathcal{U},\mathcal{V}$.

Building on this, 
another fundamental difference between Tukey and Rudin-Keisler reducibilities is that  there are cases when Fubini products  of an ultrafilter with itself are contained within the Tukey class of an ultrafilter.
For any ultrafilter with maximum Tukey type, its Fubini product with itself is also Tukey maximal, hence in the same Tukey type.

In Section 4 of \cite{Dobrinen/Todorcevic11}, it was shown that comparing Tukey types of ultrafilters the directed partial order $(\om^{\om},\le)$ was useful for gaining information about Fubini products.
(In this context, $\om^{\om}$ is thought of as the collection of all functions from $\om$ into $\om$, and $\le$ is the dominance relation, meaning $f\le g$ if and only if for all $n\in\om$, $f(n)\le g(n)$.)
In particular, the following equivalence was shown.
\begin{thm}[Dobrinen/\Todorcevic, Theorem 35 in \cite{Dobrinen/Todorcevic11}]\label{thm.DT35}
The following are equivalent for a p-point $\mathcal{U}$.
\begin{enumerate}
\item
$(\mathcal{U},\contains)\ge_T(\om^{\om},\le)$;
\item
$(\mathcal{U},\contains)\equiv_T \Pi_{n<\om}\mathcal{U}$, with the cartesian product order;
\item
$\mathcal{U}\equiv_T\mathcal{U}\cdot\mathcal{U}$.
\end{enumerate}
\end{thm}
It follows that if $\mathcal{U}$ is a rapid p-point, then $\mathcal{U}\cdot\mathcal{U}\equiv_T\mathcal{U}$.
Since $\mathcal{U}\cdot\mathcal{U}$ is never a p-point, and since the isomorphism class of a p-point consists only of p-points,
it follows that
every p-point $\mathcal{U}$ is strictly Rudin-Keisler  below
 $\mathcal{U}\cdot\mathcal{U}$.
On the other hand, it was  shown in Theorem 38 in \cite{Dobrinen/Todorcevic11} that, assuming $\mathfrak{p}=\mathfrak{c}$, there is a p-point $\mathcal{U}$ such that $\mathcal{U}\not\ge_T\om^{\om}$ and hence $\mathcal{U}<_T\mathcal{U}\cdot\mathcal{U}<_T [\mathfrak{c}]^{<\om}$.
This led to
 Question 39  in \cite{Dobrinen/Todorcevic11}, which asked whether there is an ultrafilter with the property that $\mathcal{U}\cdot\mathcal{U}<_T
\mathcal{U}\cdot\mathcal{U}\cdot\mathcal{U}$.
Milovich answered this questions strongly in the negative by proving 
the following theorem.

\begin{thm}[Milovich, Theorem 5.2 in \cite{Milovich12}]\label{thm.Milovich5.2}
For all nonprincipal filters $\mathcal{F},\mathcal{G}$ on $\om$,
$\mathcal{F}\cdot\mathcal{G}\equiv_T\mathcal{F}\cdot\mathcal{G}\cdot\mathcal{G}$.
In particular,  that for all nonprincipal filters, $\mathcal{F}\cdot\mathcal{F}\equiv_T \mathcal{F}\cdot\mathcal{F}\cdot\mathcal{F}$.
\end{thm}
Independently, Blass (unpublished) proved that for all ultrafilters $\mathcal{U}$, 
we have that $\mathcal{U}\cdot\mathcal{U}\equiv_T \mathcal{U}\cdot\mathcal{U}\cdot\mathcal{U}$.

In  \cite{Dobrinen/Todorcevic11}, Dobrinen and \Todorcevic\ proved that whenever $\mathcal{U}$ and $\mathcal{V}$ are rapid p-points,
then $\mathcal{U}\cdot\mathcal{V}\equiv_T\mathcal{V}\cdot\mathcal{U}$.
 Milovich extended this in \cite{Milovich12} to the case when $\mathcal{U}$ and $\mathcal{V}$ are simply p-filters.

\begin{thm}[Milovich, Theorem 5.4 in \cite{Milovich12}]\label{thm.Milovichreversible}
It $\mathcal{F}$ and $\mathcal{G}$ are nonprincipal filters on $\om$ and $\mathcal{G}$ is a p-filter,
then $\mathcal{F}\cdot\mathcal{G}\equiv_T \mathcal{F}\times\mathcal{G}\times \om^{\om}$.
Therefore, if $\mathcal{F}$ is also a p-filter,
then $\mathcal{F}\cdot\mathcal{G}\equiv_T\mathcal{G}\cdot\mathcal{F}$.
\end{thm}

%*****************************************************************
\section{Strictly below the maximum Tukey type}\label{sec.<top}

As Isbell  showed, there are always (in ZFC) ultrafilters with the maximum Tukey type among directed partial orders of cardinality continuum.  
In this section we investigate conditions which guarantee that an ultrafilter  be  Tukey non-maximal.
We begin with the following combinatorial characterization 
of the maximum type, which
was pointed out to us by \Todorcevic, and is proved in Fact 12 in \cite{Dobrinen/Todorcevic11}.
Let $\mathcal{U}$ be an ultrafilter.
$(\mathcal{U},\contains)\equiv_T([\mathfrak{c}]^{<\om},\sse)$ if and only if there is a subset $\mathcal{X}\sse\mathcal{U}$ such that $|\mathcal{X}|=\mathfrak{c}$ and for each infinite $\mathcal{Y}\sse\mathcal{X}$,
$\bigcap\mathcal{Y}\not\in\mathcal{U}$.
Taking the contrapositive, we find the following combinatorial characterization of not having maximum Tukey type.

\begin{fact}\label{fact.Tukeytopchar}
  $(\mathcal{U},\contains)<_T([\mathfrak{c}]^{<\om},\sse)$ if and only if for each subset $\mathcal{X}\sse\mathcal{U}$ of cardinality $\mathfrak{c}$, there is an infinite subset $\{Y_n:n<\om\}\in[\mathcal{X}]^{\om}$ such that
$\bigcap_{n<\om}Y_n\in\mathcal{U}$.
\end{fact}
This characterization aids greatly in the study of ultrafilters which are not Tukey maximum, as we shall see below.

 Milovich was the first to begin a serious investigation the structure of ultrafilters below the Tukey top.
The partial ordering $\contains$ on an ultrafilter $\mathcal{U}$ corresponds to considering the members of $\mathcal{U}$ as representing a neighborhood base of $\mathcal{U}$ in  $\beta\om$, the Stone-\v{C}ech compactification of $\om$.

\begin{thm}[Milovich, Theorem 3.11 in \cite{Milovich08}]\label{thm.Mildiamond}
Assume $\lozenge(E^{\mathfrak{c}}_{\om})$ and $\mathfrak{p}=\mathfrak{c}$.
Then there exists a nonprincipal ultrafilter $\mathcal{U}$ such that $\mathcal{U}$ is not a p-point and $(\mathfrak{c},\le)<_T(\mathcal{U},\contains)<_T([\mathfrak{c}]^{<\om},\sse)$.
\end{thm}

We now turn our attention to properties of ultrafilters which guarantee that they are not Tukey maximal.
If one examines the construction in the proof of Theorem \ref{thm.Mildiamond}, one finds that the  non-top ultrafilter  which Milovich  constructs is actually a Fubini product of $\om$ many p-points, and hence is basically generated (see Definition \ref{defn.bg} below).
The following special notion of partial ordering was introduced by Solecki and \Todorcevic\ in Section 3 of  \cite{Solecki/Todorcevic04}.

\begin{defn}[\cite{Solecki/Todorcevic04}]\label{defn.basic}
Let $D$ be a separable metric space and let $\le$ be a partial ordering on $D$.  We say that $(D,\le)$ is \em basic \rm if
\begin{enumerate}
\item
Each pair of elements of $D$ has the least upper bound with respect to $\le$ and the binary operation of least upper bound from $D\times D$ to $D$ is continuous;
\item
Each bounded sequence has a converging subsequence;
\item
Each converging sequence has a bounded subsequence.
\end{enumerate}
\end{defn}

Each ultrafilter is a separable metric space using the metric inherited from $\mathcal{P}(\om)$, where $\mathcal{P}(\om)$ is viewed as the Cantor space by identifying subsets of $\om$ with their characteristic functions on domain $\om$.
In this context, a sequence $(U_n)_{n<\om}$ of elements of $\mathcal{P}(\om)$ is said to {\em converge} to $U\in\mathcal{P}(\om)$ if and only if for each $m$ there is some $k$ such that for each $n\ge k$, $U_n\cap m=U\cap m$.
Since every bounded subset of an ultrafilter has a convergent subsequence, an ultrafilter is basic (in the general sense of a partially ordered set) if and only if (3) holds.
Hence,
an ultrafilter $\mathcal{U}$ is basic
if  and only if for every countable collection $U_n\in\mathcal{U}$ which converges
to some member $U$ in $\mathcal{U}$,
there is a subsequence $(n_k)_{k<\om}$ such that $\bigcap_{k<\om}U_{n_k}$ is a member of $\mathcal{U}$.
The next fact  is immediate from the definition and Fact \ref{fact.Tukeytopchar}.

\begin{fact}\label{fact.basicnottop}
Every basic ultrafilter is strictly below the maximum Tukey type.
\end{fact}

The next theorem shows that in fact, basic ultrafilters have already been widely studied, simply under the name of p-point.

\begin{thm}[Dobrinen/\Todorcevic, Theorem 14 in \cite{Dobrinen/Todorcevic11}] \label{thm.stevosummer}
An ultrafilter is basic if and only if it is a p-point.
\end{thm}

Thus, p-points are always below the top Tukey type.

We pause here to point out another important aspect of Tukey theory.
It turns out that Tukey reducibility is exactly the right notion to characterize p-points from among $\mathfrak{c}$-OK points.

It follows from  Proposition 3.7 in \cite{Milovich08} 
that 
if $\mathcal{U}$ is $\mathfrak{c}$-OK and not a p-point, then $(\mathcal{U},\contains)\equiv_T([\mathfrak{c}]^{<\om},\sse)$.
Combining this result  with the above fact that p-points are never Tukey maximum,
we obtain the following characterization of when a $\mathfrak{c}$-OK point is actually a p-point.

\begin{cor}
Let $\mathcal{U}$ be  a $\mathfrak{c}$-OK ultrafilter.
Then $\mathcal{U}$
is a p-point if and only if 
$(\mathcal{U},\contains)<_T([\mathfrak{c}]^{<\om},\sse)$.
\end{cor}

One can glean from the proof of Theorem \ref{thm.stevosummer}  that an ultrafilter is basic if and only if every sequence which converges to $\om$ has a bounded subsequence.
Taking the main idea from the definition of basic, namely that convergent sequences have a bounded subsequence, and relativizing this notion to a base, we obtain the weaker notion a basically generated ultrafilter.

\begin{defn}[Definition 15 in \cite{Dobrinen/Todorcevic11}] \label{defn.bg}
An ultrafilter $\mathcal{U}$  is \em basically generated \rm if there is  a filter basis $\mathcal{B}\sse\mathcal{U}$ (i.e.\ $\forall A\in\mathcal{U}$ $\exists B\in\mathcal{B}$ $B\sse A$) 
with the property that each sequence 
$\{A_n:n<\om\}\sse\mathcal{B}$ converging to an element of $\mathcal{B}$
has a subsequence $\{A_{n_k}:k<\om\}$ such that $\bigcap_{k<\om}A_{n_k}\in\mathcal{U}$.
\end{defn}

Since every basic ultrafilter is basically generated, it automatically holds that every p-point is basically generated.  In fact, the class of basically generated ultrafilters contains many non-p-points.

\begin{thm}[Dobrinen/\Todorcevic, Theorem 16 in \cite{Dobrinen/Todorcevic11}]\label{thm.2}
The collection of basically generated ultrafilters contains all p-points and all countable iterations of Fubini products of p-points.
Furthermore, it contains all countable iterations of Fubini products of basically generated ultrafilters with witnessing base which is closed under finite intersections.
\end{thm}

Since for any ultrafilter $\mathcal{U}$, its Fubini product with itself $\mathcal{U}\cdot\mathcal{U}$ is not a p-point, it follows that  there are basically generated ultrafilters which are not p-points.

The property of being basically generated was designed to retain enough strength of the property of being basic to
satisfy the combinatorial characterization  of Fact \ref{fact.Tukeytopchar}, and thus be below the top Tukey type.

\begin{thm}[Dobrinen/\Todorcevic, Theorem 33 in \cite{Dobrinen/Todorcevic11}]\label{thm.4}
If $\mathcal{U}$ is a basically generated ultrafilter on $\om$, then $\mathcal{U}<_T[\mathfrak{c}]^{<\om}$.
\end{thm}

It is still unknown whether the classes of basically generated ultrafilters  and the iterated Fubini products of p-points coincide.

\begin{problem}[Question 26 in \cite{Dobrinen/Todorcevic11}]\label{prob.bg=Fubppoint?}
Is there a basically generated ultrafilter which is not (isomorphic to) some countable iteration of Fubini products of p-points?
\end{problem}

We now turn our attention to ultrafilters on base set $\FIN$.
Recall that $\FIN$ denotes the collection of all non-empty finite subsets of the natural numbers.
In \cite{Blass87}, Blass adapted a proof of Glazer to show the following.

\begin{thm}[Blass, Theorem 2.1, \cite{Blass87}]\label{thm.BlassGlazer}
Let $\mathcal{V}_0$ and $\mathcal{V}_1$ be a pair of nonprincipal ultrafilters on $\om$.
Then there is an idempotent ultrafilter $\mathcal{U}$ on $\FIN$ such that $\mathcal{U}_{\min}=\mathcal{V}_0$ and $\mathcal{U}_{\max}=\mathcal{V}_1$.
\end{thm}

Taking any $\mathcal{V}_0=\mathcal{V}_1\equiv_T\mathcal{U}_{\mathrm{top}}$,
it follows that there exist idempotent ultrafilters on $\FIN$ realizing the maximal Tukey type.
On the other hand, there is an analogue of basic for this context, leading to ultrafilters on $\FIN$ which are not Tukey maximal.

The collection of block-sequences carry the following partial ordering $\le$.
For two infinite block-sequences $X=(x_n)_{n<\om}$ and $Y=(y_n)_{n<\om}$, define
$Y\le X$ if and only if each member of $Y$ is a finite union of elements of $X$; i.e.\ $y_n\in [X]$ for each $n$.
We write $Y\le^* X$ to mean that $Y/m\le X$ for some $m<\om$.
That is, $Y\le^* X$ if and only if there is some $k$ such that for all $n\ge k$, $y_n\in[X]$.

\begin{defn}[Definition 66 in \cite{Dobrinen/Todorcevic11}]\label{def.block-basic}
For infinite block sequences $X_n=(x^n_k)_{k<\om}$ and $X=(x_k)_{k<\om}$, the sequence
$(X_n)_{n<\om}$ {\em converges} to $X$ (written $X_n\ra X$ as $n\ra\infty$)
if for each $l<\om$ there is an $m<\om$ such that for all $n\ge m$ and all $k\le l$,
$x_k^n=x_k$.

A block-generated ultrafilter $\mathcal{U}$ is {\em block-basic} if
whenever we are given a sequence $(X_n)_{n<\om}$ of infinite block sequences of elements of $\FIN$ such that each $[X_n]\in\mathcal{U}$ 
and $(X_n)_{n<\om}$ converges to some infinite block sequence $X$ such that $[X]\in\mathcal{U}$,
then there is an infinite subsequence $(X_{n_k})_{k<\om}$ such that
$\bigcap_{k<\om}[X_{n_k}]\in\mathcal{U}$.
\end{defn}

The next fact follows immediately from the combinatorial characterization of Fact \ref{fact.Tukeytopchar}.

\begin{fact}\label{fact.bb}
If $\mathcal{U}$ is a block-basic ultrafilter on $\FIN$, then $(\mathcal{U},\contains)<_T([\mathfrak{c}]^{<\om},\sse)$.
\end{fact}

In the study of block-basic ultrafilters, the known  Ramsey theory  for $\FIN$ is useful.
The reader interested in further aspects of this is referred to \cite{Blass87}, \cite{Dobrinen/Todorcevic11}, and  \cite{TodorcevicBK10}.
Here, we simply state some equivalents of being block-basic.
Blass uses the terminology {\em stable ordered union ultrafilter} in \cite{Blass87} for what we call block-basic.
In the following theorem, the equivalence of (1) and (2) is shown in Theorem 68 in \cite{Dobrinen/Todorcevic11}, and the equivalence of (1) and (3) is shown in Theorem 4.2 of \cite{Blass87}.

\begin{thm}\label{thm.blockbasicequiv}
The following are equivalent for a block-generated ultrafilter $\mathcal{U}$ on $\FIN$.
\begin{enumerate}
\item
$\mathcal{U}$ is block-basic.
\item
For every sequence $(X_n)$ of infinite block sequences of $\FIN$ such that $[X_n]\in\mathcal{U}$ and $X_{n+1}\le^* X_n$ for each $n$,
there is an infinite block sequence $X$ such that $[X]\in\mathcal{U}$ and $X\le^* X_n$ for each $n$.
\item
$\mathcal{U}$ has the  Ramsey property.
\end{enumerate}
\end{thm}

It follows from the previous theorem that 
if $\mathcal{U}$ is a block-basic ultrafilter on $\FIN$, 
then $\mathcal{U}_{\min}$ and $\mathcal{U}_{\max}$ are 
Rudin-Keisler incomparable selective ultrafilters on $\om$.
Then, from Corollary 12, \cite{Raghavan/Todorcevic12},
we have that for any block-basic ultrafilter $\mathcal{U}$ on $\FIN$,
$\mathcal{U}_{\min}$ and $\mathcal{U}_{\max}$ are Tukey-incomparable.
Applying Theorem 2.4, \cite{Blass87} of Blass,
we get a sort of converse:
Assuming CH,
for every pair $\mathcal{V}_0$ and $\mathcal{V}_1$ of non-isomorphic selective ultrafilters on $\om$,
there is a block-basic ultrafilter $\mathcal{U}$ on $\FIN$ such that $\mathcal{U}_{\min}=\mathcal{V}_0$ and $\mathcal{U}_{\max}=\mathcal{V}_1$.

\begin{rem}
The notion of block-basic ultrafilter in fact extends to any ultrafilter generated by a topological Ramsey space, producing ultrafilters with various partition properties.
All such ultrafilters are below the Tukey top.
See Section \ref{sec.Ramsey} for more on this.  
\end{rem}

We mention that there is another type of ultrafilter which is below the Tukey maximum.
The generic ultrafilter $\mathcal{G}_2$ forced by $\mathcal{P}(\om^2)/\Fin^{\otimes 2}$ is not a p-point.
It is shown in  \cite{Blass/Dobrinen/Raghavan13} that  
 $\mathcal{G}_2$ is in fact not basically generated, 
but is strictly below the Tukey maximum, and moreover
$(\mathcal{G}_2,\contains)\not\ge ([\aleph_1]^{<\om},\sse)$.

We conclude this section with some tools which may be useful for solving Isbell's Problem.
As shown above, p-points partially ordered by either of $\contains$ and $\contains^*$ are strictly below the Tukey top. 
Combining this with Proposition 3.12 of Milovich in \cite{Milovich08}, we have  that there is an ultrafilter $\mathcal{U}$ such that $(\mathcal{U},\contains)<_T([\mathfrak{c}]^{<\om},\sse)$ if and only if there is an ultrafilter $\mathcal{V}$ such that 
$(\mathcal{V},\contains^*)<_T([\mathfrak{c}]^{<\om},\sse)$.
Since 
CH implies the existence of p-points, we will now  give attention to what happens under the assumption $\neg$CH.
Assuming $\neg$CH, the following combinatorial principle holds.

\begin{defn}[\Todorcevic, Definition 79 in \cite{Dobrinen/Todorcevic11}]\label{defn.diamond}
$\lozenge_{[\mathfrak{c}]^{\om}}$ is the statement: There exist sets $S_A\sse A$, $A\in[\mathfrak{c}]^{\om}$, such that for each $X\sse\mathfrak{c}$, 
$\{A\in[\mathfrak{c}]^{\om}:X\cap A=S_A\}$ is stationary in $[\mathfrak{c}]^{\om}$.
\end{defn}

This implies the next combinatorial principle  in the same way that the standard $\lozenge$ implies $\lozenge^-$.

\begin{defn}[\Todorcevic,  Definition 80 in \cite{Dobrinen/Todorcevic11}]\label{defn.diamondminus}
$\lozenge^-_{[[\om]^{\om}]^{\om}}$ is the statement:
There exist ordered pairs $(\mathcal{U}_A,\mathcal{X}_A)$,
where $A\in[[\om]^{\om}]^{\om}$ and $\mathcal{X}_A\sse\mathcal{U}_A\sse A$,
such that for each pair $(\mathcal{U},\mathcal{X})$ with $\mathcal{X}\sse\mathcal{U}$ and $\mathcal{X},\mathcal{U}\in[[\om]^{\om}]^{\mathfrak{c}}$,
$\{A\in[[\om]^{\om}]^{\om}:\mathcal{U}_A=\mathcal{U}\cap A$, $\mathcal{X}_A=\mathcal{X}\cap A\}$ is stationary in $[[\om]^{\om}]^{\om}$.
\end{defn}

The principle $\lozenge^-_{[[\om]^{\om}]^{\om}}$  can be used to give sufficient conditions for being below the Tukey top as well as to characterize the property of being a p-point.
Fix a $\lozenge^-_{[[\om]^{\om}]^{\om}}$ sequence
$(\mathcal{U}_A,\mathcal{X}_A)$,
where $A\in[[\om]^{\om}]^{\om}$.

\begin{defn}[Definition 81 in \cite{Dobrinen/Todorcevic11}]\label{defn.DA}
Let $P_A=\{W\in[\om]^{\om}:\exists X\in\mathcal{U}_A(W\cap X=\emptyset)\}$,
$Q_A=\{W\in[\om]^{\om}: \forall X\in\mathcal{X}_A(W\sse^*X)\}$,
and
$Q'_A=\{W\in[\om]^{\om}:\exists(B_n)_{n<\om}\in[\mathcal{X}_A]^{\om}(\forall n<\om,\ W\sse^* B_n)\}$.
Note that $Q_A\sse Q'_A$.
Let $D_A=P_A\cup Q_A$ and  $D'_A=P_A\cup Q'_A$.
\end{defn}

The following are Facts 82, 83,  and 84 in \cite{Dobrinen/Todorcevic11}.

\begin{fact}\label{fact.DAdense}
\begin{enumerate}
\item
For each $A\in[[\om]^{\om}]^{\om}$,
$D_A$ and $D'_A$ are   dense open in the partial ordering $([\om]^{\om},\contains)$.
\item
For any nonprincipal ultrafilter $\mathcal{U}$, 
$\{A\in[[\om]^{\om}]^{\om}:\mathcal{U}\cap D'_A\ne\emptyset\}$ is stationary.
\item
If $\mathcal{U}$ is an ultrafilter and $\mathcal{U}\cap D'_A\ne\emptyset$ for club many $A\in[[\om]^{\om}]^{\om}$,
then $(\mathcal{U},\contains^*)<_T([\mathfrak{c}]^{<\om},\sse)$.
Thus, there is an ultrafilter $\mathcal{V}$ such that $(\mathcal{V},\contains)<_T([\mathfrak{c}]^{<\om},\sse)$.
\end{enumerate}
\end{fact}

Thus, we obtain a new characterization of p-points, under $\neg$CH.

\begin{fact}[Dobrinen/\Todorcevic, Fact 85 in \cite{Dobrinen/Todorcevic11}]\label{fact.ppointhitsall}
The following are equivalent:
\begin{enumerate}
\item
 $\mathcal{U}$ is a p-point;
\item
 $\mathcal{U}\cap D_A\ne\emptyset$ for all $A\in[[\om]^{\om}]^{\om}$;
\item
$\mathcal{U}\cap D_A\ne\emptyset$ for club many $A$.
\end{enumerate}
\end{fact}

\begin{question}
Can Facts \ref{fact.DAdense} and \ref{fact.ppointhitsall} be applied to solve Isbell's Problem?
\end{question}

%*****************************************************************************
%*****************************************************************************

\section{Canonical cofinal maps}\label{sec.canonical.maps}

For certain classes of ultrafilters, one only need consider cofinal maps which are {\em canonical} in some sense.
Types of canonical maps found so far consist of  continuous, basic (see Definitions \ref{defn.basicom} and ), and finitely generated (see Definition \ref{defn.finitelygen}), all of which are definable.
As there are only continuum many such maps on any ultrafilter, the existence of canonical maps 
reduces the number of cofinal maps one need consider from $2^{\mathfrak{c}}$ to $\mathfrak{c}$.
This  implies that the Tukey type of any ultrafilter with canonical cofinal maps, and any ultrafilter Tukey reducible to it, has cardinality continuum,
thus marking  a strong dividing line between ultrafilters with canonical maps and those which are Tukey maximal.

The following gives an overview of the types of canonical maps which have been found so far. 
More precise statements follow later in this section.

\begin{enumerate}
\item
All ultrafilters Tukey reducible to any p-point have continuous cofinal maps (Corollary \ref{cor.DT}).
\item
All iterated Fubini products of p-points have basic,  hence finitely generated, maps (see Definition \ref{defn.basiconfront} and Theorem \ref{thm.allFubProd_p-point_cts}).
\item
All basically generated ultrafilters have finitary Tukey reductions (see Definition \ref{defn.finitelygen} and Theorem \ref{lem.16R}).
\item
All stable ordered union ultrafilters and their Rudin-Keisler $\min-\max$ projections have continuous Tukey reductions (Theorem \ref{thm.FINcanonical}).
\end{enumerate}

Up to the writing of this article,
every type of ultrafilter known to be not Tukey maximal has been shown to have canonical maps of some sort,
raising the following open problem.

\begin{question}\label{q.nottop_canon?}
Suppose  $(\mathcal{U},\contains)<_T([\mathfrak{c}]^{<\om},\sse)$.
Does it follow that  $\mathcal{U}$ has definable Tukey reductions?
Does it follow that $\mathcal{U}$ has finitely generated Tukey reductions?
\end{question}

We begin by discussing continuity of maps on ultrafilters.
By identifying subsets of $\om$ with their characteristic functions on domain $\om$,
we can consider $\mathcal{P}(\om)$ as a topological space, identifying it with the Cantor space $2^{\om}$.
Any subset of $\mathcal{P}(\om)$ can be considered as a   topological space, with the  topology inherited as a subspace of  the Cantor space.
Given  subsets $\mathcal{X},\mathcal{Y}\sse\mathcal{P}(\om)$,
a function $f:\mathcal{X}\ra\mathcal{P}(\om)$ is {\em continuous} if it is continuous with respect to the 
 subspace topologies on $\mathcal{X}$ and $\mathcal{Y}$.
A sequence $(X_n)_{n<\om}$ of members of $\mathcal{X}$ is said to converge to $X\in\mathcal{X}$ if 
there is some sequence $k_n$ such that
for all $m\ge k_n$, $X_m\cap n=X_{k_n}\cap n$.
Thus,
 a function $f:\mathcal{X}\ra\mathcal{Y}$ is continuous if for each sequence $(X_n)_{n<\om}\sse\mathcal{X}$ which converges to some $X\in\mathcal{X}$,
the sequence $(f(X_n))_{n<\om}$ converges to $f(X)$.

Continuous cofinal maps are crucial to   the analysis of 
the structure of the Tukey types of p-points, especially which structures embed into the Tukey types of p-points (e.g.\ in \cite{Dobrinen/Todorcevic11} and
\cite{Raghavan/Todorcevic12}), which we shall survey in Section \ref{sec.embeddings}.
In \cite{Raghavan/Todorcevic12}, Raghavan showed that continuous cofinal maps suffice to prove that the Tukey order is the same as the Rudin-Blass order when the reduced ultrafilter is a q-point.
Continuous cofinal maps have been emlpoyed to completely classify Tukey types and  the Rudin-Keisler structures inside them
for Ramsey ultrafilters  (in work of \Todorcevic\ in
\cite{Raghavan/Todorcevic12}), and for large classes of  rapid p-points (in   work in \cite{Dobrinen/Todorcevic12} and \cite{Dobrinen/Todorcevic13}), as will be surveyed in Section \ref{sec.Ramsey}.

It becomes useful now to fix the following notation.
For  $X$ in an ultrafilter $\mathcal{U}$,  we let $\mathcal{U}\re X$  denote $\{Y\in\mathcal{U}:Y\sse X\}$.
Note that $\mathcal{U}\re X$ is a filter base for $\mathcal{U}$, so $(\mathcal{U},\contains)\equiv_T(\mathcal{U}\re X,\contains)$.
The following theorem was the first canonization theorem for cofinal maps on ultrafilters.
It is the key for all subsequent study of the structure of Tukey types of p-points.

\begin{thm}[Dobrinen/\Todorcevic, Theorem 20 in \cite{Dobrinen/Todorcevic11}]\label{thm.5}
Suppose $\mathcal{U}$ is a p-point on $\om$ and $\mathcal{V}$ is an arbitrary ultrafilter on $\om$ such that $\mathcal{U}\ge_T \mathcal{V}$.
For each monotone cofinal map  $f:\mathcal{U}\ra\mathcal{V}$,
 there is an $X\in\mathcal{U}$
such that $f\re (\mathcal{U}\re X)$ is continuous.
Moreover, there is 
  a continuous monotone map
$f^*:\mathcal{P}(\om)\ra\mathcal{P}(\om)$ 
such that $f^*\re(\mathcal{U}\re X)=f\re(\mathcal{U}\re X)$.
Hence, there is a continuous monotone cofinal map $f^*\re\mathcal{U}$  from $\mathcal{U}$ into $\mathcal{V}$ which extends $f\re(\mathcal{U}\re X)$.
\end{thm}

In fact, the theorem gives a strong form of continuity, which we call basic.
The following definitions from \cite{Dobrinen10} fine-tune the notions of continuity and the notion of being generated by a finitary map.
To simplify notation, we identify sets with their characteristic functions.
For a set $X$, $\chi_X$ denotes its characteristic function.
For sets or characteristic functions $s,t$, we let $s\sqsubset t$ denote that $s$ is an initial segment of $t$.
(For ultra-precise statements, see \cite{Dobrinen10}.)

\begin{defn}[Dobrinen, Definitions 6 and 7 in \cite{Dobrinen10}]\label{defn.basicom}
Given  a subset $D$ of $2^{<\om}$,
we shall call a map $\hat{f}:D\ra 2^{<\om}$ \em level preserving \rm if 
there is a strictly increasing sequence  $(k_m)_{m<\om}$ such that 
for each  $s\in D\cap 2^{k_m}$, we have that $\hat{f}(s)\in 2^m$.
A level preserving map $\hat{f}$ is \em initial segment preserving \rm if whenever $m<m'$,  $s\in D\cap 2^{k_m}$, and  $s'\in D\cap 2^{k_{m'}}$, then  $s\sqsubseteq s'$ implies $\hat{f}(s)\sqsubseteq\hat{f}(s')$.
$\hat{f}$ is 
{\em monotone} if for each $s,t\in D$, $s\sse t$ implies $\hat{f}(s)\sse \hat{f}(t)$.

A monotone map $f$ on a subset $\mathcal{D}\sse \mathcal{P}(\om)$ is said to be {\em basic} if  $f$ is 
generated  by a monotone, level and initial segment preserving map in the following manner:
There is a strictly
increasing sequence $(k_m)_{m<\om}$ such that,
letting 
$D=\{\chi_X\re k_m: X\in\mathcal{D},\ m<\om\}$,
there is a 
 level and initial segment preserving map
$\hat{f}:D\ra 2^{<\om}$ 
such that 
for each $X\in\mathcal{D}$,
\begin{equation} 
f(X)=\bigcup_{m<\om} \hat{f}(\chi_X\re k_m).
%f(X)=\bigcup\{\hat{f}(s):s\in D \mathrm{\ and\ } s\sqsubseteq X\}.
\end{equation}
In this case, we say that $\hat{f}$ \em generates \rm $f$.

We say that $\mathcal{U}$ {\em has basic Tukey reductions} if for every monotone cofinal map $f:\mathcal{U}\ra\mathcal{V}$, $f$ is basic when restricted to some  cofinal subset $\mathcal{D}\sse\mathcal{U}$.
We say that
$\mathcal{U}$ \em has continuous Tukey reductions \rm if whenever  $f:\mathcal{U}\ra\mathcal{V}$ is a monotone cofinal map, there is a cofinal subset $\mathcal{D}\sse\mathcal{U}$ such that $f\re\mathcal{D}$ is continuous. 
\end{defn}

\begin{rem}
It follows from the definition that every  basic map $f$ on a subset $\mathcal{D}$ of $\mathcal{P}(\om)$ is  continuous on $\mathcal{D}$. 
Thus, if $\mathcal{U}$ has basic Tukey reductions, the $\mathcal{U}$ has continuous Tukey reductions.
\end{rem}

Using this terminology, Theorem \ref{thm.5} can be restated as follows.

\begin{thm}[Dobrinen/\Todorcevic, \cite{Dobrinen/Todorcevic11}]\label{thm.ppt}
If $\mathcal{U}$ is a p-point on $\om$, then $\mathcal{U}$ has  basic Tukey reductions.
\end{thm}

It was shown in \cite{Dobrinen10} that the property of having basic Tukey reductions is inherited under Tukey reducibility.
We point out that this is the only property known so far to be inherited under Tukey reducibility, 
whereas many standard properties, such as being a  p-point or selective, are inherited under Rudin-Keisler reducibility but not under Tukey reducibility.

\begin{thm}[Dobrinen, Theorem 9 in \cite{Dobrinen10}]\label{thm.D9}
Suppose that $\mathcal{U}$ has basic Tukey reductions.
If  $\mathcal{W}$ is Tukey reducible to $\mathcal{U}$, then $\mathcal{W}$ also has basic Tukey reductions.
\end{thm}

Theorems \ref{thm.ppt} and \ref{thm.D9} combine to show that, assuming the existence of p-points, there is a large class of ultrafilters closed under Tukey reducibility which have basic, and hence continuous, Tukey reductions.

\begin{cor}\label{cor.DT}
If $\mathcal{U}$ is Tukey reducible to a  p-point,
then $\mathcal{U}$ has basic Tukey reductions.
\end{cor}

This raises the question of whether the class of ultrafilters Tukey reducible to a p-point is the same as or strictly contained in the class of ultrafilters which have basic Tukey reductions.

\begin{question}\label{q.ctsnotppoint}
Are there ultrafilters on base set $\om$ which have basic Tukey reductions and are not Tukey reducible to any p-point?
\end{question}

The proof of Theorem \ref{thm.D9} used the following Extension Theorem.

\begin{thm}[Dobrinen, Theorem 8 in \cite{Dobrinen10}]\label{thm.PropACtsMaps}
Suppose $\mathcal{U}$ and $\mathcal{V}$ are  ultrafilters,
$f:\mathcal{U}\ra\mathcal{V}$
is a  monotone cofinal map, and
there is a cofinal subset $\mathcal{D}\sse\mathcal{U}$ such that
  $f\re\mathcal{D}$ is basic.
Then there is a continuous, monotone $\tilde{f}:\mathcal{P}(\om)\ra\mathcal{P}(\om)$ 
such that
\begin{enumerate}
\item
 $\tilde{f}$
is basic on $\mathcal{P}(\om)$;
% and is generated by
%a monotone, level and initial segment preserving map on $2^{<\om}$.
\item
 $\tilde{f}\re\mathcal{D}= f\re\mathcal{D}$; and
 \item $\tilde{f}\re\mathcal{U}:\mathcal{U}\ra\mathcal{V}$ is a cofinal map.
\end{enumerate}

Thus, $\mathcal{U}$ has basic Tukey reductions if and only if 
for every  monotone cofinal map $f:\mathcal{U}\ra\mathcal{V}$ there is some cofinal $\mathcal{D}\sse\mathcal{U}$ for which $f\re\mathcal{D}$ is basic.
\end{thm}

 The Extension Theorem says  that every basic Tukey reduction on a cofinal subset of an ultrafilter can be extended to a basic function on all of $\mathcal{P}(\om)$.
Key to this proof is that basic functions are level-preserving.
A run-of-the-mill continuous function on a cofinal subset of an ultrafilter does not a priori have to be level-preserving, since the cofinal set is not  compact. 
This leads to the question of whether or not basic and continuous Tukey reductions are the same.

\begin{question}\label{q.basicvscts}
Is there an ultrafilter on $\om$ which has continuous Tukey reductions but not basic Tukey reductions?
\end{question}

Recall that on the base set $\FIN$, there is a notion of stable ordered union ultrafilter which is the analogue of p-point for the space  of block sequences.
These ultrafilters also have continuous Tukey reductions.

\begin{thm}\label{thm.FINcanonical}
Suppose $\mathcal{U}$ is a block-basic ultrafilter on $\FIN$
and $\mathcal{V}$ is any ultrafilter on a countable index set $I$.
\begin{enumerate}
\item
(Theorem 71 in \cite{Dobrinen/Todorcevic11})
If  $\mathcal{U}\ge_T\mathcal{V}$,
then there is a monotone continuous map $f:\mathcal{P}(\FIN)\ra\mathcal{P}(I)$ such that
$f"\mathcal{U}$ is a cofinal subset of $\mathcal{V}$.
\item
(Theorem 72 in  \cite{Dobrinen/Todorcevic11})
If $\mathcal{U}_{\min,\max}\ge_T\mathcal{V}$,
then there are an infinite block sequence $\tilde{X}$ such that $[\tilde{X}]\in\mathcal{U}$ and
a monotone continuous function $f$ from $\{[X]_{\min,\max}:X\le \tilde{X}\}$ into $\mathcal{P}(I)$
whose restriction to $\{[X]_{\min,\max}:X\le \tilde{X},\ [X]\in\mathcal{U}\}$
has cofinal range in $\mathcal{V}$.
\end{enumerate}
\end{thm}

We believe that this will be true for all ultrafilters associated with any of the topological Ramsey spaces $\FIN_k^{[\infty]}$.

\begin{question}
Does every  ultrafilter on $\FIN_k$ which is block basic for  $\FIN_k^{[\infty]}$ have continuous cofinal maps?
\end{question}

 Raghavan showed in \cite{Raghavan/Todorcevic12} that continuous Tukey reductions  suffice to prove the surprising fact that the Tukey and Rudin-Blass orders sometimes coincide.
Recall the Rudin-Blass ordering: $\mathcal{V}\le_{RB}\mathcal{U}$ if and only if there is a finite-to-one map $g:\om\ra\om$ such that $\mathcal{V}=g(\mathcal{U})$.
Note that the Rudin-Blass order is stronger than the Rudin-Keisler order; that is, if $\mathcal{V}\le_{RB}\mathcal{U}$, then $\mathcal{V}\le_{RK}\mathcal{U}$.

\begin{thm}[Raghavan, Theorem 10 in \cite{Raghavan/Todorcevic12}]\label{thm.RT}
Let $\mathcal{U}$ be any ultrafilter and let $\mathcal{V}$ be a q-point.
Suppose $f:\mathcal{U}\ra\mathcal{V}$ is continuous, monotone, and cofinal in $\mathcal{V}$.
Then  $\mathcal{V}\le_{RB}\mathcal{U}$.
\end{thm}

Corollary \ref{cor.DT} and
Thorem \ref{thm.RT} combine to  yield the following important picture of the structure of Tukey types of q-points.

\begin{cor}\label{cor.DRT}
Suppose $\mathcal{W}$ is Tukey reducible to a p-point.
Then  every q-point Tukey reducible to $\mathcal{W}$ is in fact Rudin-Blass reducible to $\mathcal{W}$ and hence is selective.
\end{cor}

This shows that at least the structure of q-points Tukey reducible to some p-point is quite simple.

There are ultrafilters which are not Tukey maximal which do not have continuous Tukey reductions.
As Raghavan shows in Corollary 11 in \cite{Raghavan/Todorcevic12},
the Fubini product of non-isomorphic ultrafilters does not have continuous Tukey reductions.
However, a weaker form of canonical maps can still be obtained for basically generated ultrafilters and other types of Tukey non-maximal ultrafilters as we review below.

\begin{defn}[Definition 3 in \cite{Dobrinen10}]\label{defn.finitelygen}
We say that an ultrafilter 
$\mathcal{U}$ on a countable base $B$ 
\em
has finitary Tukey reductions \rm if the following holds:
For every ultrafilter $\mathcal{V}$ on a countable base $C$, whenever $f:\mathcal{U}\ra\mathcal{V}$ 
is a monotone cofinal map,
there is
 a cofinal subset $\mathcal{D}\sse\mathcal{U}$
and 
 a  function $\hat{g}: [B]^{<\om}\ra [C]^{<\om}$,
such that 
\begin{enumerate}
\item
$\hat{g}$ is monotone: $s\sse t\ra \hat{g}(s)\sse \hat{g}(t)$; 
 and
\item
$\hat{g}$ generates $f$ on $\mathcal{D}$:
For each $X\in\mathcal{D}$,
$f(X)=\bigcup_{k<\om}\hat{g}(X\cap B\re k)$,
\end{enumerate}
where $B\re k$ denotes the first $k$ members of $B$ under some fixed enumeration of $B$ in order type $\om$.
A map $f$ generated by such a $\hat{g}$ is called {\em finitely generated}.
\end{defn}

Note that every continuous map is finitary, but not vice versa.
The qualitative difference is that a finitary map may very well not be level-preserving.
Still, the property of having finitary Tukey reductions guarantees that the ultrafilter is not Tukey maximal:
It is clear that there are only continuum many finitely generated cofinal maps on a given ultrafilter;
so if an ultrafilter has finitary Tukey reductions, then its Tukey type, and the Tukey type of every ultrafilter Tukey reducible to it, has cardinality continuum.

In \cite{Raghavan/Todorcevic12},
Raghavan  proved that basically generated ultrafilters have finitary Tukey reductions.

\begin{thm}[Raghavan, 
Lemmas 8 and  16 of \cite{Raghavan/Todorcevic12}]\label{lem.16R}
Let $\mathcal{U}$ be basically generated by a base $\mathcal{B}$.
Let  $\phi:\mathcal{B}\ra\mathcal{V}$ be  a monotone  map, and 
define $\psi_{\phi}(A)=\bigcap\{\phi(X):X\in\mathcal{B}$ and $X\sse A\}$.
Then there is some $U\in\mathcal{U}$ such that 
$\psi_{\phi}(A)$ is finite for all $A\in\mathcal{B}\re U$.
Moreover, for every $A\in\mathcal{B}$,
$\bigcup_{s\in[B]^{<\om}}\psi_{\phi}(s)\ne\emptyset$.
\end{thm}

He then  used this to show that Tukey reduction can be framed as Rudin-Keisler reduction in the following sense.
Given an ultrafilter $\mathcal{U}$ on $\om$ and a subset $P\sse\FIN$, define $\mathcal{U}(P)$ to be 
$\{A\sse P:\exists X\in \mathcal{U}( P\cap [X]^{<\om}\sse A)\}$.
If for all $X\in\mathcal{U}$, $P\cap [X]^{<\om}$ is infinite,
then $\mathcal{U}(P)$ is a proper, non-principal filter on base set $P$.

\begin{thm}[Raghavan, Theorem 17 in \cite{Raghavan/Todorcevic12}]\label{thm.RbgRK}
Let $\mathcal{U}$ be basically generated by base $\mathcal{B}\sse\mathcal{U}$.
Let $\mathcal{V}$ be an arbitrary ultrafilter.
Then there is a subsets $P\sse\FIN$ such that 
\begin{enumerate}
\item
For all $s,t\in P$, $s\sse t$ implies $s=t$;
\item
$\mathcal{U}(P)\equiv_T\mathcal{U}$; and
\item
$\mathcal{V}\le_{RK}\mathcal{U}(P)$.
\end{enumerate}
\end{thm}

Thus, every ultrafilter Tukey reducible to some basically generated ultrafilter $\mathcal{U}$ is in fact Rudin-Keisler below a canonically constructed ultrafilter Tukey equivalent to $\mathcal{U}$.

Closing the class of p-points under countable Fubini products produces the class of ultrafilters which we call {\em Fubini iterates of p-points}. 
Recall that every Fubini iterate of p-points is basically generated, and it is open whether or not every basically generated ultrafilter is a Fubini iterate of p-points.
Fubini iterates of p-points have canonical maps which are not only finitary but moreover 
satisfy the analogue of
basic when considered on the correct topological space.

The definition of {\em basic} for Fubini iterates of p-points is analagous to the definition for ultrafilters on $\om$.
As it is useful both here and in the final section of this paper, we now define the notion of a front, which can be found in \cite{TodorcevicBK10}.

\begin{defn}\label{def.front}
A family $B$ of finite subsets of some infinite subset $I$ of $\om$ is called a \em front \rm on $I$ if 
\begin{enumerate}
\item
$a\not\sqsubset b$ whenever $a, b$ are in $B$; and
\item
For every infinite $X\sse I$ there exists $b\in B$ such that $b\sqsubset X$.
\end{enumerate}
\end{defn}

The following special kind of front is exactly the fronts that form base sets for Fubini iterates.

\begin{defn}[Dobrinen, Definition 13 in \cite{Dobrinen10}] \label{def.flattopfront}
We  call a set $B\sse[\om]^{<\om}$ a {\em flat-top front} if
$B$ is a front on $\om$, $B\ne \{\emptyset\}$, and
\begin{enumerate}
\item
Either  $B= [\om]^1$; or
\item
$B\sse [\om]^{\ge 2}$ and 
for each $b\in B$, 
letting $a=b\setminus\{\max(b)\}$,
$\{c\setminus a: c\in B,\ c\sqsupset a\}$ is equal to $[\om\setminus (\max(a) +1)]^1$.
\end{enumerate}
\end{defn}

Flat-top fronts are exactly the fronts on which iterated Fubini products of ultrafilters are represented.  
For example, $[\om]^2$ is the flat-top front on which a Fubini product of the form $\lim_{n\ra\mathcal{U}}\mathcal{V}_n$ is represented.
For each $k<\om$,  $[\om]^k$ is a flat-top front.
Moreover,  
flat-top fronts are preserved under the following recursive construction, which is analogous to  the base of a Fubini product of ultrafilters which are themselves Fubini products:
Given
  flat-top fronts $B_{\{n\}}$ on $\om\setminus (n+1)$, $n<\om$,
the union $\bigcup_{n\in \om} B_n$ is a flat-top front on $\om$.

Given a flat-top front $B$, let $\hat{B}$ denote the collection of all initial segments of members of $B$, and $C(B)$ denote $\hat{B}\setminus B$.  Note that $C(B)$ forms a tree, where the tree ordering is by initial segment.
A sequence $\vec{\mathcal{U}}=(\mathcal{U}_c:c\in C(B))$ of nonprincipal ultrafilters $\mathcal{U}_c$ on $\om$,
a {\em $\vec{\mathcal{U}}$-tree} is a tree $T\sse \hat{B}$ with the property that $\{n\in\om:c\cup \{n\}\in T\}\in\mathcal{U}_c$ for all $c\in C$.

\begin{fact}\label{fact.precise.connection}
If $\mathcal{W}$  is  a countable  iteration of  Fubini products of p-points, then there is a flat-top front $B$ and p-points $\mathcal{U}_c$, $c\in C(B)$ such that $\mathcal{W}$ is isomorphic to the ultrafilter on $B$ generated by the $(\mathcal{U}_c:c\in C)$-trees.
\end{fact}

\begin{notation}\label{notn.nec}
For any subset $A\sse[\om]^{<\om}$ and $k<\om$, 
let $A\re k$ denote $\{a\in A:\max(a)<k\}$.
For $A\sse\hat{B}$ and $k<\om$,
let $\chi_A\re k$ denote the characteristic function of $A\re k$ on domain $\hat{B}\re k$.
For each $k<\om$, let $2^{\hat{B}\re k}$ denote the collection of characteristic functions of subsets of $\hat{B}\re k$ on domain $\hat{B}\re k$.

Given a flat-top front $B$ and a sequence $\vec{\mathcal{U}}=(\mathcal{U}_c:c\in C(B))$ of nonprincipal ultrafilters on $\om$,
let $\mathfrak{T}=\mathfrak{T}(\vec{\mathcal{U}})$ denote the collection of all $\vec{\mathcal{U}}$-trees.
For any tree $T$, let $[T]$ denote the collection of maximal branches through $T$.
\end{notation}

\begin{defn}\label{defn.finrep}
Let $B$ be a flat-top front on $\om$. 
Let  $(n_k)_{k<\om}$ be
an increasing sequence.
We say that a function $\hat{f}:\bigcup_{k<\om}2^{\hat{B}\re n_k}\ra 2^{<\om}$ is 
{\em level preserving} if  $\hat{f}: 2^{\hat{B}\re {n_k}}\ra 2^{k}$, for each $k<\om$.
$\hat{f}$ is {\em initial segment preserving} if for all $k<m$, $A\sse\hat{B}\re n_k$ and $A'\sse \hat{B}\re n_m$,
if $A=A'\re n_k$ then $\hat{f}(\chi_A)=\hat{f}(\chi_{A'})\re k$.
$\hat{f}$ is {\em monotone} if whenever  $A\sse A'\sse\hat{B}$ are finite,
 then  $\hat{f}(s)\sse \hat{f}(t)$.

Let $\mathcal{W}$ be an ultrafilter on $B$ generated by $(\mathcal{U}_c:c\in C(B))$-trees, let $f:\mathcal{W}\ra\mathcal{V}$ be a monotone cofinal map, where $\mathcal{V}$ is an ultrafilter on base $\om$, and let $\tilde{T}\in\mathfrak{T}(\vec{\mathcal{U}})$.
We say that 
$\hat{f}:\bigcup_{k<\om}2^{\hat{B}\re n_k}\ra 2^{<\om}$ {\em generates} $f$ on $\mathfrak{T}\re \tilde{T}$ if
for each   $T\sse\tilde{T}$  in $\mathfrak{T}$,
\begin{equation}
f([T])=\bigcup_{k<\om}d(\hat{f}(\chi_T \re n_k)).
\end{equation}
\end{defn}

\begin{defn}\label{defn.basiconfront}
Let $B$ be a flat-top front and $\vec{\mathcal{U}}=(\mathcal{U}_c:c\in C(B))$ be a sequence of ultrafilters.
We say that the ultrafilter $\mathcal{W}$ on $B$ generated by the  $\vec{\mathcal{U}}$-trees has {\em basic Tukey reductions}
if whenever $f:\mathcal{W}\ra\mathcal{V}$ is  a monotone cofinal map,
then there is a $\tilde{T}\in\mathfrak{T}(\vec{\mathcal{U}})$ and a monotone, initial segment and level preserving map $\hat{f}$ which generates $f$ on $\mathfrak{T}\re \tilde{T}$.
\end{defn}

In particular, a basic Tukey reduction is finitary.
Moreover,
if $\hat{f}$ witnesses that $f$ is basic on $\mathfrak{T}\re\tilde{T}$, 
then $\hat{f}$ generates a continuous map on the collection of {\em trees} in $\mathfrak{T}\re\tilde{T}$,  continuity being with respect to the Cantor topology on $2^{\hat{B}}$.

\begin{thm}[Dobrinen, Theorem 21 in \cite{Dobrinen10}]\label{thm.allFubProd_p-point_cts}
Let $B$ be any flat-top front and $\vec{\mathcal{U}}=(\mathcal{U}_c:c\in C(B))$ be a sequence of p-points.
Then the ultrafilter on base $B$ generated by the $\vec{\mathcal{U}}$-trees
has basic Tukey reductions.
Therefore, every countable iteration of Fubini products of p-points has basic, and hence finitary Tukey reductions.
\end{thm}

Since the previous theorem is  the direct analogue of Theorem \ref{thm.ppt}, Theorem \ref{thm.D9} points to the following natural open question.

\begin{question}\label{q.bgimplyinherited}
If $\mathcal{W}$ is Tukey reducible to a Fubini iterate of p-points, does $\mathcal{W}$ have finitary Tukey reductions?
\end{question}

The above results leave open the following question.
Existing methods should suffice to give a positive answer, but the details have not yet been carried out.

\begin{question}\label{q.ccmq}
Given an ultrafilter Tukey reducible to a countable iteration of Fubini products of ultrafilters
from among the collection listed above, 
is every monotone cofinal map on $\mathcal{U}$ generated by a finitary map on some cofinal subset?
\end{question}

In his study of which ultrafilters are Tukey above a selective ultrafilter,
Raghavan showed, among other things, the following interesting theorem.

\begin{thm}[Raghavan, Corollary 56 in \cite{Raghavan/Todorcevic12}]\label{thm.R56}
Let $\mathcal{U}$ be a Fubini iterate of p-points and $\mathcal{V}$ be a selective ultrafilter.
If $\mathcal{V}\le_T\mathcal{U}$,
then in fact $\mathcal{V}\le_{RK}\mathcal{U}$.
\end{thm}

This is part of the larger investigation in Section 7 of \cite{Raghavan/Todorcevic12} of the following question.

\begin{question}
What is Tukey above a selective ultrafilter? 
\end{question}

In particular, Raghavan asks the following, which he answered for certain circumstances.

\begin{question}
If $\mathcal{U}$ is basically generated and $\mathcal{V}$ is selective, then does $\mathcal{V}\le_T\mathcal{U}$ imply $\mathcal{V}\le_{RK}\mathcal{U}$?
\end{question}

This section concludes 
by mentioning  that there are similar 
 canonization theorems for cofinal maps on generic ultrafilters forced by $\mathcal{P}(\om\times\om)/\Fin^{\otimes 2}$.
(See \cite{Blass/Dobrinen/Raghavan13}.)

%*********************************************************************
%*********************************************************************

\section{Structures embedded in the Tukey types}\label{sec.embeddings}

We now review structures known to embed into the Tukey types of ultrafilters.
All of the results in this section have heavily relied on the existence of 
 canonical cofinal maps presented in Section \ref{sec.canonical.maps}.
In this section we concentrate on structures which are just embedded into the Tukey types, saving the known results for exact structures for the next section.

That p-points have continuous Tukey reductions (recall Theorem \ref{thm.5}) is heavily used in the study of which structures embed into the Tukey types of p-points.
It is shown in Corollary 21 in \cite{Dobrinen/Todorcevic11} that every $\le_T$-chain of p-points on $\om$ has cardinality $\le\mathfrak{c}^+$.
In fact, this is true for any type of ultrafilter which has canonical cofinal maps, since in that case, there are only $\mathfrak{c}$ many ultrafilters Tukey reducible to any ultrafilter in the chain.
Thus, by the theorems of Section \ref{sec.canonical.maps}, we have the following.

\begin{fact}
Let $(C,\le_T)$ be any Tukey-increasing chain of ultrafilters from among ultrafilters which are  basically generated, stable ordered union ultrafilters, or  Tukey reducible to a p-point, or any sort of ultrafilter which has only $\mathfrak{c}$ many ultrafilters Tukey reducible to it.
Then $C$  has cardinality at most $\mathfrak{c}^+$.
\end{fact}

The following was  proved by Dobrinen and Todorcevic, and  independently by Raghavan.

\begin{thm}[Dobrinen/\Todorcevic, 
  Corollary 53 in \cite{Dobrinen/Todorcevic11}]\label{cor.53}
 Assuming CH,  there is a Tukey strictly increasing chain of p-points of order type $\mathfrak{c}$.
\end{thm}

Embeddings of antichains into the Tukey types of p-points has also been studied.
The first fact we mention here follows immediately from continuous Tukey reductions and the Hajnal Free Set Theorem.
 
\begin{fact}[Dobrinen/\Todorcevic, Corollary 23 in \cite{Dobrinen/Todorcevic11}]\label{cor.5.5}
Every family $\mathcal{X}$ of p-points on $\om$ of cardinality $>\mathfrak{c}^+$ contains a subfamily $\mathcal{Z}\sse\mathcal{X}$ of equal size such that $\mathcal{U}\not\le_T\mathcal{V}$ whenever $\mathcal{U}\ne\mathcal{V}$ are in $\mathcal{Z}$.
\end{fact}

In the next theorem, almost minimal assumptions are used to show that antichains of maximal size embed into the Tukey types of p-points and furthermore, into the Tukey types of selective ultrafilters.

\begin{thm}[Dobrinen/\Todorcevic, Theorem 44 in \cite{Dobrinen/Todorcevic11}]
\label{thm.selective}\ \\
\begin{enumerate}
\item
Assume cov$(\mathscr{M})=\mathfrak{c}$. 
Then there are $2^{\mathfrak{c}}$ pairwise Tukey incomparable selective ultrafilters.
\item
Assume  $\mathfrak{d}=\mathfrak{u}=\mathfrak{c}$.
Then there are $2^{\mathfrak{c}}$ pairwise Tukey incomparable p-points.
\end{enumerate}
\end{thm}

The structure of the Tukey types of p-points turns out to embed interesting  configurations.
For example, the diamond configuration embeds into the Tukey types of p-points.

\begin{thm}[Dobrinen/\Todorcevic, Theorem 57 in \cite{Dobrinen/Todorcevic11}]
Assuming Martin's Axiom, there is a p-point with two Tukey-incomparable  Tukey predecessors, which in turn have a common Tukey lower bound.
\end{thm}

This section concludes with the known facts about ultrafilter on base set $\FIN$.
Recall that these are not technically p-points, but are in fact analogues of selective ultrafilters for the Milliken space.
Hindman's Theorem and the canonical maps from  Theorem \ref{thm.FINcanonical}are used to obtain the following structure theorem.

\begin{thm}[Dobrinen/\Todorcevic, Theorem 57  in \cite{Dobrinen/Todorcevic11}]\label{thm.5Tukey}
Assuming CH, there is a block-basic ultrafilter $\mathcal{U}$ on $\FIN$ such that $\mathcal{U}_{\min,\max}<_T\mathcal{U}$ and  
$\mathcal{U}_{\min}$ and $\mathcal{U}_{\max}$ are Tukey incomparable. 
\end{thm}

\begin{question}\label{q.UUmin}
If $\mathcal{U}$ is any block-basic ultrafilter, does it follow that $\mathcal{U}>_T\mathcal{U}_{\min,\max}$?
\end{question}

The next section continues the investigation of structures which embed into the Tukey types of p-points in  fine-tuned manner.

%***********************************************************
%***********************************************************

\section{Near the bottom of the Tukey hierarchy:\\
 Initial structures and connections with Ramsey theory}\label{sec.Ramsey}

The  structure of the Tukey types of ultrafilters becomes completely clear, as  does the 
relationship between the Rudin-Keisler and Tukey reducibilities 
near the bottom of the Tukey hierarchy.
It is well-known that  Ramsey ultrafilters are minimal in the Rudin-Keisler hierarchy.
In \cite{Raghavan/Todorcevic12}, \Todorcevic\ showed that the analogous result holds in the Tukey hierarchy.
(As usual, we only consider nonprincipal ultrafilters, which is why  Ramsey ultrafilters are called  minimal.)
In fact, he proved a much finer result, strengthening a previous result in \cite{Dobrinen/Todorcevic11} which showed that every ultrafilter Tukey reducible to a Ramsey ultrafilter is basically generated.
The following is an equivalent re-statement of Theorem 24 in \cite{Raghavan/Todorcevic12}.

\begin{thm}[\Todorcevic, Theorem 24 in \cite{Raghavan/Todorcevic12}]\label{thm.TRamsey}
Let $\mathcal{U}$ be a Ramsey ultrafilter.
If $\mathcal{V}\le_T\mathcal{U}$, then either $\mathcal{V}$ is isomorphic to a Fubini iterate of $\mathcal{U}$, or else $\mathcal{V}$ is principal.
\end{thm}

This theorem provides  a sharp result relating Rudin-Keisler and Tukey reducibilities with several important structural consequences.
To begin, 
Theorem \ref{thm.TRamsey} shows that the Tukey type of a Ramsey ultrafilter $\mathcal{U}$ consists exactly of the Rudin-Keisler types of ultrafilters which are countable iterations of Fubini products of $\mathcal{U}$.
It also  follows that  Ramsey ultrafilters are minimal in the Tukey hierarchy, since every Fubini iterate of a Ramsey ultrafilter is Tukey equivalent to that ultrafilter (recall Theorem \ref{thm.DT35})
Since Fubini products of p-points are never p-points,
we see that, up to isomorphism, there is exactly one p-point in the Tukey type of $\mathcal{U}$, namely  $\mathcal{U}$ itself.
\Todorcevic's proof makes essential use of continuous Tukey reductions for p-points and a Ramsey-classification theorem of \Pudlak\ and \Rodl\ (see Theorem \ref{thm.PR} below).

Given an ultrafilter $\mathcal{U}$, we use the terminology {\em initial Tukey structure below $\mathcal{U}$}
to refer to the structure of all Tukey types of nonprincipal ultrafilters Tukey reducible to $\mathcal{U}$.
The above result then says that the initial Tukey structure below a Ramsey ultrafilter is a singleton, consisting simply of the Tukey type of the Ramsey ultrafilter itself.

We  shall provide an outline of the proof of Theorem \ref{thm.TRamsey}, as it is instructive 
for gaining a feel for the sorts of proofs of the other theorems in this section.
The {\em Ellentuck space} is the triple $([\om]^{\om},\sse,r)$, where $r$ is the finitization function.
Members $X\in[\om]^{\om}$ are considered as infinite increasing sequences of natural numbers, $X=\{x_0,x_1,x_2,\dots\}$.
For each $n<\om$, the $n$-th approximation to $X$ is $r_n(X)=\{x_i:i<n\}$; in particular, $r_0(X)=\emptyset$.
The basic open sets of the Ellentuck topology are sets of the form 
$[a,X]=\{Y\in [\om]^{\om}: a\sqsubset Y$ and $Y\sse X\}$.
Thus, the Ellentuck topology is finer than the metric topology on $[\om]^{\om}$.

The Ellentuck space satisfies the following important Ramsey property:
Whenever a subset $\mathcal{X}\sse[\om]^{\om}$
has the property of Baire in the Ellentuck topology, 
then that set is {\em Ramsey},
meaning that every open set contains a basic open set either contained in $\mathcal{X}$ or else disjoint from $\mathcal{X}$.
Topological Ramsey spaces are topological spaces which generalize the Ellentuck space in the sense that every subset with the property of Baire is Ramsey.
Rather than provide all relevant background here, we refer the reader interested in general topological Ramsey spaces to \cite{TodorcevicBK10}.

A {\em front} on the Ellentuck space is a collection  $\mathcal{F}\sse[\om]^{<\om}$ such that
(a) For each $X\in[\om]^{\om}$, there is an $a\in \mathcal{F}$ for which $a\sqsubset X$; and (b)
For all $a\ne b$ in $\mathcal{F}$, $a\not\sqsubset b$.
A map $\vp$ from a front into $\om$ is called {\em irreducible} if 
(a) $\vp$ is {\em inner}, meaning that $\vp(a)\sse a$ for all $a\in\mathcal{F}$; and
(b) $\vp$ is {\em Nash-Williams}, meaning that $\vp(a)\not\sqsubset\vp(b)$ for all $a,b\in\mathcal{F}$ such that $\vp(a)\ne\vp(b)$.
A  {\em barrier} is a front which is also {\em Sperner}, meaning that for all $a\ne b$ in $\mathcal{F}$, $a$ is not a proper subset of $b$.

Given a front $\mathcal{F}$ and an $X\in[\om]^{\om}$,
we let $\mathcal{F}\re X$ denote $\{a\in \mathcal{F}:a\sse X\}$.
Given an equivalence relation $E$ on a barrier $\mathcal{F}$,
we say that an irreducible map $\vp$ {\em represents} $E$ on $\mathcal{F}\re X$ if for all $a,b\in\mathcal{F}\re X$, we have $a\ E\ b\lra \vp(a)=\vp(b)$.

The following theorem of \Pudlak\ and \Rodl\ is the basis for all subsequent canonization theorems for fronts on general topological Ramsey spaces considered in the papers
\cite{Dobrinen/Todorcevic12} and \cite{Dobrinen/Todorcevic13}.

\begin{thm}[\Pudlak/\Rodl, \cite{Pudlak/Rodl82}]\label{thm.PR}
For any barrier $\mathcal{F}$ on the Ellentuck space and any equivalence relation on $\mathcal{F}$,
there is an $X\in[\om]^{\om}$ and an irreducible map $\vp$ such that
the equivalence relation restricted to $\mathcal{F}\re X$ is represented by $\vp$.
\end{thm}

The proof of Theorem \ref{thm.TRamsey} proceeds as follows.
Let $\mathcal{U}$ be a Ramsey ultrafilter and let $\mathcal{V}\le_T\mathcal{U}$.
Since $\mathcal{U}$ is in particular a p-point, there is a continuous cofinal map witnessing this Tukey reduction, say $f:\mathcal{U}\ra\mathcal{V}$.
Since $f$ is continuous, it is generated by a monotone (level and initial segment preserving) finitary map $\hat{f}:[\om]^{<\om}\ra[\om]^{<\om}$; that is, for each $U\in\mathcal{U}$, $f(U)=\bigcup_{k<\om}\hat{f}(U\cap k)$.
For each $U\in\mathcal{U}$, take the minimal finite initial segment $a\sqsubset U$ for which $\hat{f}(a)\ne\emptyset$.
The collection of all such $a$ form a front on $\mathcal{U}$, denoted $\mathcal{F}$.
Now define the function $g:\mathcal{F}\ra \om$ by $g(a)=\min(\hat{f}(a))$.
Then $g$ induces an equivalence relation on $\mathcal{F}$.
It follows from $\mathcal{U}$ being Ramsey that there is an $X\in\mathcal{U}$ and an irreducible function $\vp$ which canonizes the equivalence relation induced by $g$ on $\mathcal{F}\re X$.

We now transfer the ultrafilter $\mathcal{U}$ on base set $\om$ to an ultrafilter on base set $\mathcal{F}\re X$ as follows.
For each $U\in\mathcal{U}$,
let $\mathcal{F}\re U=\{a\in \mathcal{F}:a\sse U\}$.
Since $\mathcal{U}$ is Ramsey, it follows that the collection of sets $\mathcal{F}\re U$, $U\in\mathcal{U}\re X$, generates an ultrafilter on base $\mathcal{F}\re X$,
 denoted as $\mathcal{U}\re\mathcal{F}\re X$.
One then proves that $\mathcal{U}\re\mathcal{F}\re X$ is Tukey equivalent to $\mathcal{U}$ and 
 that $g(\mathcal{U}\re\mathcal{F}\re X)$ is in fact equal to $\mathcal{V}$.
The fact that $\vp$ is irreducible is then used to show that $\mathcal{V}$ is Rudin-Keisler equivalent to an iterated Fubini power of $\mathcal{U}$.

This concludes our sketch of the proof of Theorem \ref{thm.TRamsey}.
The proofs of the theorems below follow this general outline, provided that there are  Ramsey-classification theorems available for the topological Ramsey spaces associated with the given ultrafilters.
The discovery of the Ramsey-classification theorems for the new topological Ramsey spaces referred to below was the key to finding the initial Tukey structures of their associated ultrafilters. 
Because of space constraints, we do not give the Ramsey-classification theorems here, but refer the interested reader to the relevant papers.

Given Theorem \ref{thm.TRamsey}, the most natural question  to ask is whether Ramsey ultrafilters are an anomaly, or whether similar theorems hold more generally for ultrafilters which are Ramsey-like.
Weakly Ramsey ultrafilters are
the most reasonable Ramsey-like ultrafilters on which to begin this line of investigation.
Every weakly Ramsey ultrafilter is  Rudin-Keisler minimal above a Ramsey ultrafilter, as was shown by Blass in 
\cite{Blass74}.
In \cite{Laflamme89}, Laflamme constructed $\sigma$-closed  partial orders $\bP_{\al}$, $1\le\al<\om_1$, which force weakly Ramsey ultrafilters ($\bP_1$) and other rapid p-points satisfying weak partition properties  ($\bP_{\al}$, $2\le\al<\om_1$).
He  proved that $\bP_{\al}$ forces an ultrafilter with Rudin-Keisler predecessors forming exactly a decreasing chain of order-type $\al+1$, the minimum ultrafilter being Ramsey.
Moreover, he showed that these forced ultrafilters
 have {\em complete combinatorics},
meaning that there are some combinatorial statements such that any ultrafilter satisfying those statements is forcing generic over HOD$(\mathbb{R})^{V[G]}$,
where $V[G]$ is the model obtained by L\'{e}vy collapsing a Mahlo cardinal to $\aleph_1$ and HOD$(\mathbb{R})^{V[G]}$ denotes the model of all sets in $V[G]$ which are hereditarily definable from ordinals and members of $\mathbb{R}$.

In \cite{Dobrinen/Todorcevic12} and \cite{Dobrinen/Todorcevic13},
Dobrinen and \Todorcevic\ extract the essential properties of Laflamme's forcings $\bP_{\al}$, $1\le\al<\om_1$,  and construct new topological Ramsey spaces $\mathcal{R}_{\al}$ forcing equivalent to $\bP_{\al}$.
Thus, for each $1\le\al<\om_1$, 
 ultrafilters $\mathcal{U}_{\al}$ associated with the space $\mathcal{R}_{\al}$ satisfy all the same partition properties as Laflamme's ultrafilters forced by $\bP_{\al}$.
The fact that the $\mathcal{R}_{\al}$ are topological Ramsey spaces puts at one's disposal the available abstract Ramsey theory from \cite{TodorcevicBK10}.
This is employed in  \cite{Dobrinen/Todorcevic12} and \cite{Dobrinen/Todorcevic13}  to prove new Ramsey-classification theorems which generalize the \Pudlak-\Rodl\ Theorem.
These theorems are crucial for allowing one to 
applying ideas from  \Todorcevic's original argument to the new topological spaces.
The new Ramsey-classification theorems are used to decode the Rudin-Keisler types within the Tukey types of the ultrafilters $\mathcal{U}_{\al}$.
From this, the structure of the Tukey types of all ultrafilters Tukey reducible to $\mathcal{U}_{\al}$ is made clear.
Though the Ramsey-classification theorems are essential to this work, in keeping the focus of this paper on Tukey, we refer the interested reader to the aforementioned papers.

We point out that
topological Ramsey spaces provide precise understanding of the mechanisms at work behind complete combinatorics.
A 
theorem of \Todorcevic\ (Theorem 4.9  in \cite{Farah98}) states that in the presence
of a supercompact  cardinal, an ultrafilter on $\omega$ is selective if and only if it is
generic for $\mathcal{P}(\om)/\Fin$ over $L(\mathbb{R})$.
Similar characterizations were recently shown for a large class of
ultrafilters forming a precise hierarchy above selective ultrafilters
(see \cite{Dobrinen/Todorcevic12}, \cite{Dobrinen/Todorcevic13},
and \cite{Mijares/Nieto13}).

We now present an overview of the known initial Tukey structures.
The following shows that the initial Tukey structures for Laflamme's ultrafilters are exactly the same as the Rudin-Keisler structures which he found.
The following theorem combines
Theorem 5.18 for $\mathcal{U}_1$ in \cite{Dobrinen/Todorcevic12} and Theorem 5.13 for $\mathcal{U}_{\al}$, $2\le \al<\om_1$, in \cite{Dobrinen/Todorcevic12}.

\begin{thm}[Dobrinen/\Todorcevic]\label{thm.DTR_1}
Given $1\le\al<\om_1$, let $\mathcal{U}_{\al}$ be an ultrafilter associated with $\mathcal{R}_{\al}$.
Then the initial structure of the Tukey below $\mathcal{U}_{\al}$ is exactly a decreasing chain of order type $(\al+1)^*$.
\end{thm}

We remark that the minimal Tukey type below $\mathcal{U}_{\al}$ is that of a Ramsey ultrafilter.

Theorem \ref{thm.DTR_1} is actually a consequence of the fact that every Fubini iterate of a rapid p-point is Tukey equivalent to itself 
and the following finer result, which extends 
Theorem \ref{thm.TRamsey} above.
Here we favor presenting a statement of the theorems understandable from the background given in this survey
rather than their full strength.
The reader is referred to the original papers for the strongest versions of the following, given in  Theorem 5.10 for $\mathcal{U}_1$ in \cite{Dobrinen/Todorcevic12} and Theorem 5.11 for $\mathcal{U}_{\al}$, $2\le \al<\om_1$, in \cite{Dobrinen/Todorcevic13}.

\begin{thm}[Dobrinen/\Todorcevic]\label{thm.DTR_2}
Let $1\le\al<\om_1$, and let $\mathcal{V}\le_T\mathcal{U}_{\al}$.
Then the Tukey type of $\mathcal{V}$ consists precisely of the isomorphism classes of ultrafilters which are iterated Fubini products of ultrafilters from among a fixed countable collection of rapid p-points.
\end{thm}

The rapid p-points in the previous theorem are exactly obtained from canonical projection maps on ultrafilters from blocks in the Ramsey space $\mathcal{R}_{\al}$.
The rapid p-points inside the Tukey type of $\mathcal{U}_1$ form a Rudin-Keisler strictly increasing chain of order-type $\om$.
This extends an earlier result of Raghavan and \Todorcevic\ in \cite{Raghavan/Todorcevic12} in which they constructed two p-points with the same Tukey class, one of which is Rudin-Keisler strictly above the other.
For $2\le\al<\om_1$, the structure of the isomorphism types of the p-points in the Tukey type of $\mathcal{U}_{\al}$ is more complex, 
but is still understood precisely from the underlying tree structure and the canonical equivalence relations.   
In particular, there are p-points in the  Tukey type of $\mathcal{U}_2$ which are Rudin-Keisler incomparable, the first example of that kind.
See  \cite{Dobrinen/Todorcevic12} and  \cite{Dobrinen/Todorcevic13} for the details.

We conclude this survey with a couple more important open problems regarding the Tukey theory of ultrafilters.

Recall that the proof of Theorem \ref{thm.5Tukey}
 shows that the generic filter for the forcing notion $(\FIN^{[\infty]},\le^*)$ adjoins a block-basic ultrafilter $\mathcal{U}$ on $\FIN$ with the properties stated in Theorem \ref{thm.blockbasicequiv}.
On the other hand, as mentioned above,
 if there is a supercompact cardinal,
then in particular, every block-basic ultrafilter $\mathcal{U}$ on $\FIN$ is generic over $L(\mathbb{R})$ for the forcing notion
$(\FIN^{[\infty]},\le^*)$.
This leads us also to the following related problem from \cite{Dobrinen/Todorcevic11}.

\begin{problem}\label{problem.LR5}
Assume the existence of a supercompact cardinal.
Let $\mathcal{U}$ be an arbitrary block-basic ultrafilter on $\FIN$.
Show that the inner model $L(\mathbb{R})[\mathcal{U}]$
has exactly five Tukey types of ultrafilters on a countable index set.
\end{problem}

We  conclude with the following variant of Isbell's Problem.
\begin{problem}\label{prob.Isbell2}
Is it possible to construct, without using extra axioms of set theory, 
an ultrafilter $\mathcal{U}$ on $\om$ whose Tukey type has cardinality continuum?
\end{problem}

\bibliographystyle{amsplain}
\bibliography{references}

\providecommand{\bysame}{\leavevmode\hbox to3em{\hrulefill}\thinspace}
\providecommand{\MR}{\relax\ifhmode\unskip\space\fi MR }
\providecommand{\MRhref}[2]{%
  \href{http://www.ams.org/mathscinet-getitem?mr=#1}{#2}
}
\providecommand{\href}[2]{#2}
\begin{thebibliography}{10}

\bibitem{Bartoszynski/JudahBK}
Tomek Bartoszy{\'{n}}ski and Haim Judah, \emph{{S}et {T}heory on the
  {S}tructure of the {R}eal {L}ine}, A. K. Peters, Ltd., 1995.

\bibitem{Blass74}
Andreas Blass, \emph{Ultrafilter mappings and their {D}edekind cuts},
  Transactions of the American Mathematical Society (1974), no.~188, 327--340.

\bibitem{Blass87}
\bysame, \emph{Ultrafilters related to {H}indman's finite-unions theorem and
  its extensions}, Contemporary Mathematics \textbf{65} (1987), 89--124.

\bibitem{BlassHB}
\bysame, \emph{Combinatorial cardinal characteristics of the continuum},
  Handbook of Set Theory (Matthew Foreman, Akihiro Kanamori, and Menachem
  Magidor, eds.), Springer, Dordrecht, 2010.

\bibitem{Blass/Dobrinen/Raghavan13}
Andreas Blass, Natasha Dobrinen, and Dilip Raghavan, \emph{The next best thing
  to a p-point},  (2013), Submitted.

\bibitem{ComfortBK74}
W.~Wistar Comfort, \emph{The {T}heory of {U}ltrafilters}, Springer-Verlag,
  1974.

\bibitem{Day44}
Mahlon~M. Day, \emph{Oriented systems}, Duke Mathematical Journal \textbf{11}
  (1944), 201--229.

\bibitem{Dobrinen10}
Natasha Dobrinen, \emph{Continuous cofinal maps on ultrafilters},  (2010),
  Submitted.

\bibitem{Dobrinen/Todorcevic11}
Natasha Dobrinen and Stevo Todor{\v{c}}evi{\'{c}}, \emph{Tukey types of
  ultrafilters}, Illinois Journal of Mathematics \textbf{55} (2011), no.~3,
  907--951.

\bibitem{Dobrinen/Todorcevic12}
\bysame, \emph{{R}amsey-{C}lassification {T}heorems and their application in
  the {T}ukey theory of ultrafilters, {P}art 1}, Transactions of the American
  Mathematical Society (2013), 27 pp. To appear.

\bibitem{Dobrinen/Todorcevic13}
\bysame, \emph{{R}amsey-{C}lassification {T}heorems and their application in
  the {T}ukey theory of ultrafilters, {P}art 2}, Transactions of the American
  Mathematical Society (2013), 30 pp. To appear.

\bibitem{Farah98}
Ilijas Farah, \emph{Semiselective coideals}, Mathematika \textbf{45} (1998),
  no.~1, 79--103.

\bibitem{Fremlin91}
David Fremlin, \emph{The partially ordered sets of measure theory and {T}ukey's
  ordering}, Note di Matematica \textbf{XI} (1991), 177--214.

\bibitem{Hindman74}
Neil Hindman, \emph{Finite sums from sequences within cells of a partition of
  $n$}, Journal of Combinatorial Theory. Series A \textbf{17} (1974), 1--11.

\bibitem{Isbell65}
John Isbell, \emph{The category of cofinal types. {II}}, Transactions of the
  American Mathematical Society \textbf{116} (1965), 394--416.

\bibitem{Isbell72}
\bysame, \emph{Seven cofinal types}, J. London Math. Soc. \textbf{2} (1972),
  no.~4, 651--654.

\bibitem{Juhasz66}
Istvan Juh{\'{a}}sz, \emph{Remarks on a theorem of {B}.\
  {P}osp{\'{i}}{\v{s}}il}, General Topology and its Relations to Modern
  Analysis and Algebra, Academia Publishing House of the Czechoslovak Academy
  of Sciences, Praha, 1967, pp.~205--206.

\bibitem{Juhasz67}
Istv{\'{a}}n Juh{\'{a}}sz, \emph{Remarks on a theorem of {B}.\
  {P}osp{\'{i}}{\v{s}}il. ({R}ussian)}, Commentationes Mathematicae
  Universitatis Carolinae \textbf{8} (1967), 231--247.

\bibitem{Kunen78}
Kenneth Kunen, \emph{Weak {P}-points in ${N}^*$}, Colloquia Mathematica
  Societatis J\'{a}nos Bolyai, 23. Topology, Budapest (1978), 741--749.

\bibitem{Laflamme89}
Claude Laflamme, \emph{Forcing with filters and complete combinatorics}, Annals
  of Pure and Applied Logic \textbf{42} (1989), 125--163.

\bibitem{Matrai10}
Tam{\'{a}}s M{\'{a}}trai, \emph{On a $\sigma$-ideal of compact sets}, Topology
  and Its Applications \textbf{157} (2010), no.~8, 1479--1484.

\bibitem{Mijares/Nieto13}
Jos{\'{e}}~G. Mijares and Jesus Nieto, \emph{Local {R}amsey theory {A}n
  abstract approach}, arXiv:0712.2393v1 (2013), 11 pp, preprint.

\bibitem{Milovich08}
David Milovich, \emph{Tukey classes of ultrafilters on $\om$}, Topology
  Proceedings \textbf{32} (2008), 351--362.

\bibitem{Milovich12}
\bysame, \emph{Forbidden rectangles in compacta}, Topology and its Applications
  \textbf{159} (2012), 3180--3189.

\bibitem{Pospisil39}
Bed{\v{r}}ich Posp{\'{i}}{\v{s}}il, \emph{On bicompact spaces}, Publ. Fac. Sci.
  Univ. Masaryk \textbf{270} (1939).

\bibitem{Pudlak/Rodl82}
Pavel Pudl{\'{a}}k and Vojtech R{\"{o}}dl, \emph{Partition theorems for systems
  of finite subsets of integers}, Discrete Mathematics \textbf{39} (1982),
  67--73.

\bibitem{Raghavan/Todorcevic12}
Dilip Raghavan and Stevo Todor{\v{c}}evi{\'{c}}, \emph{Cofinal types of
  ultrafilters}, Annals of Pure and Applied Logic \textbf{163} (2012), no.~3,
  185--199.

\bibitem{Schmidt55}
J{\"{u}}rgen Schmidt, \emph{Konfinalit{\"{a}}t}, Zeitschrift f{\"{u}}r
  Mathematische Logik und Grundlagen der Mathematik \textbf{1} (1955),
  271--303.

\bibitem{Solecki/Todorcevic04}
Slawomir Solecki and Stevo Todor{\v{c}}evi{\'{c}}, \emph{Cofinal types of
  topological directed orders}, Annales de L'Institut Fourier \textbf{54}
  (2004), no.~6, 1877--1911.

\bibitem{Solecki/Todorcevic11}
\bysame, \emph{Avoiding families and {T}ukey functions on the nowhere-dense
  ideal}, Journal of the Institute of Mathematics of Jussieu \textbf{10}
  (2011), no.~2, 405--435.

\bibitem{TodorcevicDirSets85}
Stevo Todor{\v{c}}evi{\'{c}}, \emph{Directed sets and cofinal types},
  Transactions of the American Mathematical Society \textbf{290} (1985), no.~2,
  711--723.

\bibitem{Todorcevic96}
\bysame, \emph{A classification of transitive relations on $\omega_1$},
  Proceedings of the London Mathematical Society \textbf{(3) 73} (1996), no.~3,
  501--533.

\bibitem{TodorcevicBK10}
\bysame, \emph{Introduction to {R}amsey {S}paces}, Princeton University Press,
  2010.

\bibitem{Tukey40}
John~W. Tukey, \emph{Convergence and uniformity in topology}, Princeton
  University Press, 1940.

\end{thebibliography}

\end{document}